\documentclass{amsart}
\usepackage{amsthm,amsfonts,amssymb,amsmath,color}

\numberwithin{equation}{section}

\renewcommand\d{\partial}

\renewcommand\b{\beta}
\renewcommand\o{\omega}

\renewcommand\t{\TT}
\newcommand\R{\mathbb R}\newcommand\N{\mathbb N}

\def\g{\gamma}
\def\de{\delta}

\def\t{\TT}
\def\O{\Omega}
\def\th{\theta}

\def\l{\lambda}
\def\vp{\varphi}
\def\epsilon{\varepsilon}
\def\e{\varepsilon}

\def\tpsi{\widetilde \psi}

\newcommand\br{\begin{rem}}
\newcommand\er{\end{rem}}
\newcommand\bp{\begin{pmatrix}}
\newcommand\ep{\end{pmatrix}}
\newcommand\be{\begin{equation}}
\newcommand\ee{\end{equation}}
\newcommand\ba{\begin{equation}\begin{aligned}}
\newcommand\ea{\end{aligned}\end{equation}}

\newcommand\nn{\nonumber}

\newcommand{\RR}{{\mathcal R}}

\newcommand{\TT}{{\mathbb T}}
\newcommand{\CC}{{\mathbb C}}

\newcommand{\II}{{\mathbb I}}

\newcommand{\SSS}{{\mathbb S}}

\newcommand{\A}{{\mathcal A}}

\newcommand{\vv}{{\mathbf v}}
\newcommand{\ff}{{\mathbf f}}

\newcommand{\vvarphi}{{\boldsymbol \varphi}}

\newcommand{\Ov}[1]{\overline{#1}}

\newcommand{\DC}{C^\infty_c}

\newcommand{\vr}{\varrho}

\newcommand{\vu}{\vc{u}}

\newcommand{\vg}{\vc{g}}
\newcommand{\vc}[1]{{\bf #1}}
\newcommand{\Div}{{\rm div}}
\newcommand{\Grad}{\nabla_x}

\newcommand{\dx}{{\rm d} {x}}
\newcommand{\dt}{{\rm d} t }
\newcommand{\dq}{{\rm d} q }

\newcommand{\intO}[1]{\int_{\O} #1 \, \dx}

\newtheorem{definition}{Definition}[section]
\newtheorem{theorem}[definition]{Theorem}

\newtheorem{lemma}[definition]{Lemma}

\newtheorem{remark}[definition]{Remark}

\numberwithin{equation}{section}

\usepackage[margin=1.9cm]{geometry}

\begin{document}

\title[Dissipative weak solutions to compressible Navier--Stokes--Fokker--Planck systems]{Dissipative weak solutions to\\ compressible Navier--Stokes--Fokker--Planck systems\\ with variable viscosity coefficients}

\author[Eduard Feireisl]{Eduard Feireisl}
\address{Institute of Mathematics of the Academy of Sciences of the Czech Republic, Zitn\' a 25, 115 67 Prague 1, Czech Republic}
\email{feireisl@math.cas.cz}

\author[Yong Lu]{Yong Lu}
\address{Mathematical Institute, Faculty of Mathematics and Physics, Charles University, Sokolovsk\'a 83, 186 75 Prague, Czech Republic}
\email{luyong@karlin.mff.cuni.cz}

\author[Endre S\"{u}li]{Endre S\"{u}li}
\address{Mathematical Institute, University of Oxford\\ Andrew Wiles Building, Woodstock Rd., Oxford OX2 6GG, UK}
\email{suli@maths.ox.ac.uk}

\subjclass[2000]{}

\date{}

\maketitle

\begin{abstract}

Motivated by a recent paper by Barrett and S\" uli [\textit{J.W. Barrett \& E. S\"uli: Existence of global weak solutions to compressible isentropic finitely extensible bead-spring chain models  for  dilute  polymers,  Math.  Models  Methods  Appl.  Sci., 26 (2016)}],
we consider the compressible Navier--Stokes system coupled with a Fokker--Planck type equation describing the motion of polymer molecules in a viscous compressible fluid occupying a bounded spatial domain, with polymer-number-density-dependent viscosity coefficients.
The model arises in the kinetic theory of dilute solutions of nonhomogeneous polymeric liquids,
where the polymer molecules are idealized as bead-spring chains with finitely extensible nonlinear elastic (FENE) type spring potentials.
The motion of the solvent is governed by the unsteady, compressible, barotropic Navier--Stokes system, where the viscosity coefficients in the Newtonian stress tensor depend on the polymer number density. Our goal is to show that the existence theory developed in the case of constant viscosity coefficients can be extended to the case of polymer-number-density-dependent viscosities, provided that certain technical restrictions are imposed, relating the behavior of the viscosity coefficients and the pressure for large values of the solvent density. As a first step in this direction, we prove here the weak sequential stability of the family of dissipative (finite-energy) weak solutions to the system.

\end{abstract}


\renewcommand{\refname}{References}


\section{Introduction}

 In \cite{Barrett-Suli}, the authors established the existence of global-in-time weak solutions to the Navier--Stokes--Fokker--Planck equations arising in the kinetic theory of dilute polymer solutions and describing a large class of bead-spring chain models with finitely extensible nonlinear elastic (FENE) type spring potentials. For $\O\subset \R^3$ a bounded domain, the solvent density $\vr$ and the solvent velocity field $\vu$ satisfy the following equations in the space-time cylinder $(0,T]\times \O$, $T>0$:
\be\label{i1}
\d_t \vr + \Div_x (\vr \vu) = 0,
\ee
\be\label{i2}
 \d_t (\vr\vu)+ \Div_x (\vr \vu \otimes \vu) +\nabla_x p(\vr) -
\Div_x \SSS =\Div_x \TT +\vr\, \ff,
\ee
which we assume here to be supplemented with the no-slip boundary condition
\be\label{i4}
\vu=\mathbf{0}\quad \mbox{on}\ (0,T]\times \partial\Omega.
\ee

Ignoring the effect of temperature changes, we consider a barotropic pressure law
\be\label{pressure}
p = p(\vr), \ p(\vr) \approx \vr^\gamma \ \mbox{for large values of}\ \vr.
\ee
The {\em Newtonian stress tensor} $\SSS$ is defined by
\be\label{i3}
\SSS  = \mu^S \left( \frac{\nabla_x \vu + \nabla^{\rm T}_x \vu}{2} - \frac{1}{3} (\Div_x \vu) \II \right) + \mu^B (\Div_x \vu) \II,
\ee
with the shear and bulk viscosity coefficients $\mu^S$ and $\mu^B$ defined below (see \eqref{mu-eta1}). In contrast with \cite{Barrett-Suli},  where the shear and bulk viscosity coefficients are taken to be constant, $\mu^S >0$ and $\mu^B \geq 0$, we consider here the case when they are functions of the polymer number density. Such an extension requires nontrivial modifications of the method used in \cite{Barrett-Suli}, and represents the main contribution of the present paper.

In \emph{a bead-spring chain model} consisting of $K+1$ beads coupled with $K$ elastic springs representing a polymer chain,  the non-Newtonian elastic extra stress tensor $\TT$ is defined by a version of the Kramers expression (cf. \eqref{def-TT0} below), depending on the probability density function $\psi$, which, in addition to $t$ and $x$, also depends on the conformation vector $(q_1^{\rm T},\dots q_K^{\rm T})^{\rm T} \in \R^{3K}$, with $q_i$ representing the $3$-component \emph{conformation/orientation vector} of the $i$th spring in the chain. Let $D:=D_1\times \cdots \times D_K \subset \R^{3K}$ be the domain of admissible conformation vectors. Typically $D_i$ is the whole space $\R^3$ or a bounded open ball centered at the origin $0$ in $\R^3$, for each $i=1,\dots,K$. When $K=1$, the model is referred to as the \emph{dumbbell model}.  Here we focus on finitely extensible nonlinear (or, briefly, FENE-type) bead-spring-chain models where $D_i=B(0,b_i^{\frac12})$, a ball centered at the origin $0$ in $\mathbb{R}^3$ and of radius $b_i^{\frac12}$, with $b_i>0$ for each $i \in \{1,\dots, K\}$. The {\em extra-stress tensor} $\TT$ is defined by the formula:
\be\label{def-TT0}
\TT (\psi)(t,x) := \TT_1 (\psi) (t,x) -\left(\int_{D\times D} \gamma(q,q')\, \psi(t,x,q)\,\psi(t,x,q') \, \dq\,\dq'\right)\II,
\ee
where, similarly to \cite{Barrett-Suli}, the interaction kernel $\g$ is assumed to be a positive constant $\gamma(q,q')\equiv \de>0.$
Consequently,
\be\label{def-TT}
\TT (\psi) := \TT_1 (\psi)  -\de \left(\int_{D}  \psi \,\dq\right)^2\II.
\ee
The first part, $\TT_1(\psi)$, of $\TT(\psi)$ is given by the {\em Kramers expression}
\be\label{def-TT1}
\t_1(\psi) := k \left[\left(\sum_{i=1}^K \CC_i(\psi)\right)-(K+1) \left(\int_{D}  \psi \ \dq\right)\II\right],
\ee
where $k>0$ is the product of the Boltzmann constant and the absolute temperature and
\be\label{def-Ci}
\CC_{i}(\psi)(t,x) := \int_D \psi(t,x,q)\, U_i'\bigg(\frac{|q_i|^2}{2}\bigg)\, q_i q_i^{\rm T}\,\dq, \quad i=1,\dots,K.
\ee

In the expression \eqref{def-Ci}, the smooth functions $U_i : [0,\frac{b_i}{2}) \to [0,\infty)$,  $i=1,\dots,K,$ are the spring potentials satisfying
$U_i(0)=0$, $\lim_{s\to \frac{b_i}{2}-} U_i(s) =+\infty$.
We introduce the {\em partial Maxwellian} $M_i : D_i \to [0,\infty)$ by
\be\label{def-Mi}
M_i(q_i) := \frac{1}{Z_i} {\rm e}^{- U_i\big(\frac{|q_i|^2}{2}\big)},\quad \mbox{where } Z_i:=\int_{D_i} \mathrm{e}^{- U_i\big(\frac{|p_i|^2}{2}\big)}\,{\rm d}p_i.
\ee
The {\em Maxwellian} $M : D \to [0,\infty)$ is then defined as the product of the $K$ partial Maxwellians: i.e., for any $q=(q_1^{\rm T},\dots, q_K^{\rm T})^{\rm T}$ in $D = D_1\times\cdots \times D_K$, we have that
\[M(q) := \prod_{i=1}^K M_i(q_i).\]
Clearly,
$
\int_D M(q)\ \dq=1.
$
By direct calculations one verifies that, for any $i \in \{1,\dots,K\}$,
\be\label{pt2-M}
M(q)\, \nabla_{q_i} (M(q))^{-1} = - M(q)^{-1}\, \nabla_{q_i} M(q) = \nabla_{q_i}\left(U_i\bigg(\frac{|q_i|^2}{2}\bigg)\right) = U_i'\bigg(\frac{|q_i|^2}{2}\bigg) q_i
\quad \forall\,q = (q_1^{\rm T},\dots,q_K^{\rm T})^{\rm T}\in D. 
\ee

As in \cite{Barrett-Suli}, we shall suppose that, for any $i\in\{1,\dots, K\}$, there exist positive constants $c_{ij}, \ j=1,\dots,4$, and $\th_i > 1$ such that
\ba\label{pt1-Mi-Ui}
c_{i1}\left({\rm dist}\,(q_i,\d D_i)\right)^{\th_i} \leq M_i(q_i)\leq c_{i2}\left({\rm dist}\,(q_i,\d D_i)\right)^{\th_i},\quad
c_{i3} \leq \left({\rm dist}\,(q_i,\d D_i)\right) U_i'\left(\frac{|q_i|^2}{2}\right)\leq c_{i4} \quad \forall\,q_i\in D_i.
\ea
It is then straightforward to deduce that
\be\label{pt3-Mi-Ui}
\int_{D_i} \left(1+ \left( U_i\bigg(\frac{|q_i|^2}{2}\bigg) \right)^2+ \left( U_i'\bigg(\frac{|q_i|^2}{2}\bigg) \right)^2  \right) M_i(q_i)\,\dq_i <\infty,\quad i=1,\dots, K.
\ee

The probability density function $\psi$ satisfies the following {\em Fokker--Planck equation} in $(0,T]\times \O\times D$:
\ba\label{eq-psi}
\d_t \psi + \Div_x (\vu\,\psi) + \sum_{i=1}^K \Div_{q_i} \left(  (\nabla_x \vu)\, q_i\, \psi \right)= \e \Delta_x \psi+\frac{1}{4\lambda} \sum_{i=1}^K\sum_{j=1}^K A_{ij}\, \Div_{q_i}\left( M \nabla_{q_j} \left( \frac{\psi}{M} \right)\right).
\ea

The \emph{centre-of-mass diffusion} term $\e \Delta_x \psi$ is generally of the form
$
\e \Delta_x \left( \frac{\psi}{\zeta(\vr)} \right),
$
which involves the {\em drag coefficient} $\zeta(\cdot)$ depending on the fluid density $\vr$. Here we assume that $\zeta$ is a constant function, which is, for simplicity, taken to be identically $1$. The constant parameter $\e$ is the {\em centre-of-mass diffusion coefficient}, which is strictly positive.
The positive parameter $\l$ is called the \emph{Deborah number}; it characterizes the elastic relaxation property of the fluid. The constant matrix $A=(A_{ij})_{1\leq i,j\leq K}$, called the \textit{Rouse matrix}, is symmetric and positive definite. We denote by $A_0$ the smallest eigenvalue of $A$; clearly, $A_0>0$. We refer to Section 1 of Barrett and S\"uli \cite{Barrett-Suli1} for a derivation of the Fokker--Planck equation \eqref{eq-psi}.

The Fokker--Planck equation needs to be supplemented by suitable boundary conditions. For any $i=1,\dots, K$, let $\d \overline D_i:= D_1\times \cdots \times D_{i-1} \times \d D_i \times D_{i+1} \cdots \times D_K$ and suppose that
\ba\label{boundary-psi}
\left(\frac{1}{4\lambda} \sum_{j=1}^K A_{ij}\, \Div_{q_i}\left(M \nabla_{q_j} \left( \frac{\psi}{M} \right)\right) -(\nabla_x \vu) \, q_i\, \psi  \right)\cdot \frac{q_i}{|q_i|}&=0 \quad &&\mbox{on}\ (0,T]\times \O \times \d \overline D_i,\\
 \nabla_x\psi \cdot {\bf n}&=0 \quad &&\mbox{on}\ (0,T]\times \d\O \times D.
\ea

Finally,  we introduce the {\em polymer number density} $\eta$ defined by
\be\label{def-eta}
\eta(t,x) := \int_D \psi(t,x,q)\,\dq, \qquad (t,x)\in [0,T]\times \O.
\ee
By (formally) integrating the partial differential equation \eqref{eq-psi} over $D$ and using the boundary condition in $\eqref{boundary-psi}_1$, and by integrating the boundary condition $\eqref{boundary-psi}_2$ over $D$, we deduce the following partial differential equation and boundary condition for the function $\eta$:
\be\label{eq-eta}
\d_t \eta + \Div_x (\vu \,\eta) = \e \Delta_x \eta\quad \mbox{in}\ (0,T] \times \O;\qquad \nabla_x \eta \cdot {\bf n}=0 \quad \mbox{on}\  (0,T]\times \d\O.
\ee
By noting \eqref{def-eta} we see that the expression for the extra-stress tensor in \eqref{def-TT} and \eqref{def-TT1} can also be expressed as follows:
\be\label{def-TT-f}
\TT (\psi) := k \left(\sum_{i=1}^K \CC_i(\psi)\right)-\left(k(K+1) \eta  + \de\,  \eta^2\right) \II.
\ee
As has already been pointed out above, our main objective is to consider a class of models of this form where the viscosity coefficients $\mu^S = \mu^S(\eta)$ and $\mu^B = \mu^B(\eta)$
depend on $\eta$.

\subsection{Dissipative (finite-energy) weak solutions}

We adopt the following hypotheses on the initial data:
\ba\label{ini-data}
&\vr(0,\cdot) = \vr_0(\cdot) \ \mbox{with}\ \vr_0 \geq 0 \ {\rm a.e.} \ \mbox{in} \ \O, \quad \vr_0 \in L^\gamma(\O) \ \mbox{with}\ \gamma>\textstyle{\frac32};\\
&\vu(0,\cdot) = \vu_0(\cdot) \in L^r(\O;\R^3) \ \mbox{for some $r>1$}\ \mbox{such that}\ \vr_0|\vu_0|^2 \in L^1(\O);\\
&\psi(0,\cdot) = \psi_0(\cdot) \ \mbox{with}\ \psi_0 \geq 0 \ \mbox{a.e. in}  \ \O\times D,\quad  \psi_0 \left(\log \frac{\psi_0}{M}\right) \in L^1(\O\times D);\\
&\eta(0,\cdot)=\int_{D}\psi_0 \,\dq =:\eta_0 \in L^2 (\O).
\ea
We deduce from $\eqref{ini-data}_1$ and $\eqref{ini-data}_2$ by using H\"older's inequality that $(\vr\vu)(0,\cdot) = \vr_0 \vu_0 =\sqrt{\vr_0}\sqrt{\vr_0} \vu_0 \in L^{\frac{2\g}{\g+1}}(\O;\R^3)$.

\begin{definition}\label{def-weaksl} We say that $(\vr,\vu,\psi,\eta)$ is a dissipative (finite-energy)
weak solution in $(0,T]\times \O\times D$ to the system of equations
\eqref{i1}--\eqref{i3},  \eqref{eq-psi}--\eqref{def-TT-f}, supplemented by the initial data \eqref{ini-data}, if:
\begin{itemize}

\item $\vr \geq 0 \ {\rm a.e.\  in} \ (0,T] \times \Omega,\quad  \vr \in  C_w ([0,T];  L^\gamma(\Omega)),\quad \vu\in L^{r}(0,T;W_0^{1,r}(\Omega; \R^3))\quad \mbox{for some $r>1$},$
\ba\label{weak-est}
&\vr \vu \in C_w([0,T]; L^{\frac{2 \gamma}{ \gamma + 1}}(\Omega; \R^3)),\quad \vr |\vu|^2 \in L^\infty(0,T; L^{1}(\Omega));\\
&\psi \geq 0 \ {\rm a.e.\  in} \ (0,T] \times \Omega \times D, \quad \psi \in C_{w}([0,T];L^1(\Omega\times D)),\\
& \nabla_x \psi\in L^1((0,T)\times \O\times D;\R^3), \quad  M \nabla_{q} \left(\frac{\psi}{M}\right)\in L^1((0,T)\times \O\times D;\R^{3K}),\\
&\eta =\int_D \psi\ \dq  \ {\rm a.e.\  in} \ (0,T] \times \Omega,\quad \eta \in C_w ([0,T];  L^2(\Omega)) \cap L^2 (0,T;  W^{1,2}(\Omega)),\\
&\TT (\psi) := k \left(\sum_{i=1}^K \CC_i(\psi)\right)-\left(k(K+1) \eta  + \de\,  \eta^2\right) \II \quad {\rm a.e.\  in} \ (0,T] \times \Omega,\quad  \TT\in L^1((0,T)\times \O;\R^{3\times 3}).
\ea

\item For any $t \in (0,T]$ and any test function $\phi \in C^\infty([0,T] \times \Ov{\Omega})$, one has

\be\label{weak-form1}
\int_0^t\intO{\big[ \vr \partial_t \phi + \vr \vu \cdot \Grad \phi \big]} \,\dt' =
\intO{\vr(t, \cdot) \phi (t, \cdot) } - \intO{ \vr_{0} \phi (0, \cdot) },
\ee
\be\label{weak-form2}
\int_0^t \intO{ \big[ \eta \partial_t \phi + \eta \vu \cdot \Grad \phi - \e \nabla_x\eta \cdot \nabla_x \phi \big]} \, \dt' =  \intO{ \eta(t, \cdot) \phi (t, \cdot) } - \intO{ \eta_{0} \phi (0, \cdot) }.
\ee

\item For any $t \in (0,T]$ and any test function $\vvarphi \in \DC([0,T] \times {\Omega};\R^3)$, one has
\ba\label{weak-form3}
&\int_0^t \intO{ \big[ \vr \vu \cdot \partial_t \vvarphi + (\vr \vu \otimes \vu) : \Grad \vvarphi  + p(\vr)\, \Div_x \vvarphi - \SSS : \Grad \vvarphi \big] } \, \dt'\\
&= \int_0^t \intO{ \TT : \nabla_x\vvarphi - \vr\, \ff \cdot \vvarphi} \, \dt' + \intO{ \vr \vu (t, \cdot) \cdot \vvarphi (t, \cdot) } - \intO{ \vr_{0} \vu_{0} \cdot \vvarphi(0, \cdot) },
\ea
where $\SSS$ is defined by \eqref{i3}.

\item For any $t \in (0,T]$ and any test function $\phi \in \DC([0,T] \times \Ov{\Omega}\times D)$, one has
\ba\label{weak-form4}
&\int_0^t \intO{\int_D \bigg[ \psi \partial_t \phi + \psi \vu \cdot \Grad \phi  + \sum_{i=1}^K (\nabla_x \vu)\, q_i\, \psi \cdot \nabla_{q_i} \phi - \e\nabla_x  \psi \cdot \nabla_x \phi \bigg]\,\dq}\,\dt'\\
&=\frac{1}{4\lambda} \sum_{i=1}^K\sum_{j=1}^K A_{ij}\int_0^t \intO{ \int_D  M\nabla_{q_j}\left(\frac{\psi}{M}\right)\cdot \nabla_{q_i}\phi\ \dq} \, \dt' + \intO{ \int_D \psi \phi (t, \cdot)- \psi_{0}  \phi(0, \cdot)\,\dq }.
\ea

\item The continuity equation holds in the sense of renormalized solutions:
\be\label{weak-renormal}
\d_t b(\vr) +\Div_x (b(\vr)\vu) + (b'(\vr)\vr - b(\vr))\,\Div_x \vu =0 \quad \mbox{in} \ \mathcal{D}' ((0,T]\times \O),
\ee
for any
\be\label{cond-b-renormal}
b\in C^1([0,\infty)),\quad |b'(s)s|+| b(s)| \leq c<\infty \qquad \forall\, s\in [0,\infty).
\ee
\item  Let $\mathcal{F}(s) := s(\log s -1)+1$ for $s>0$ and define $\mathcal{F}(0):=\lim_{s\to 0+} \mathcal{F}(s)=1$. For a.e. $t \in (0,T]$, the
following \emph{energy inequality} holds:
\ba\label{energy}
&\int_{\Omega} \bigg[
\frac{1}{2} \vr |\vu|^2 +  P(\vr) +\de\, \eta^2 + k \int_D M \mathcal{F} (\widetilde \psi)\ \dq\bigg] (t, \cdot) \, \dx\\
&\quad+ \int_0^t \intO{ \mu^S \bigg|\frac{\nabla \vu + \nabla^{\rm T} \vu}{2} - \frac{1}{3} (\Div_x \vu)\, \II\, \bigg|^2 +\mu^B |\Div_x \vu|^2  } \, \dt' + 2\, \e\, \de \int_0^t \intO{   |\nabla_x \eta|^2} \, \dt'\\
&\quad+ \e\, k \int_0^t \intO{ \int_D M\bigg|\nabla_x \sqrt{\widetilde \psi}\, \bigg|^2 \, \dq} \,\dt'
+ \frac{k\,A_0}{4\lambda} \int_0^t \intO{ \int_D M \bigg|\nabla_{q}  \sqrt{\widetilde \psi}\,\bigg|^2 \, \dq} \, \dt'\\
&\leq \int_{\Omega} \bigg[ \frac{1}{2} \vr_{0} |\vu_{0} |^2 +  P(\vr_0) + \de\, \eta_0^2 + k \int_D M \mathcal{F} \bigg(\frac{\psi_0}{M}\bigg)\, \dq\bigg]  \dx + \int_0^t\int_\O \vr \,\ff \cdot \vu \,\,\dx\, \dt',
\ea
 where we have set $P(\vr) := \vr \int_1^\vr p(z)/ z^2 \, {\rm d}z$ and $\widetilde \psi:=\frac{\psi}{M}$.

\end{itemize}

\end{definition}

\begin{remark} Definition \ref{def-weaksl} is fairly standard. The energy
inequality (\ref{energy}) identifies an important class of weak solutions, usually termed \emph{dissipative (finite-energy)}.
We note that, given a \emph{smooth} solution, by tedious but rather straightforward calculations one can obtain the following a priori bound (see (1.22) in \cite{Barrett-Suli}):
\ba\label{energy0}
&\int_{\Omega} \left[
\frac{1}{2} \vr |\vu|^2 +  P(\vr) +\de\, \eta^2 + k \int_D M \mathcal{F} (\widetilde \psi)\, \dq\right] (t, \cdot) \, \dx+ \int_0^t \intO{ \SSS : \Grad \vu\, } \, \dt' + 2\, \e\, \de \int_0^t \intO{   |\nabla_x \eta|^2}\, \dt'\\
& \quad + \e\, k \int_0^t \intO{ \int_D M\left|\nabla_x \sqrt{\widetilde \psi} \,\right|^2 \dq} \, \dt' + \frac{k}{4\lambda} \sum_{i=1}^K\sum_{j=1}^K A_{ij}\int_0^t \intO{ \int_D M \nabla_{q_j}  \sqrt{\widetilde \psi} \cdot \nabla_{q_i}  \sqrt{\widetilde \psi} \, \dq} \, \dt' \\
&\leq \int_{\Omega} \left[
\frac{1}{2} \vr_{0} |\vu_{0} |^2 +  P(\vr_0) + \de\, \eta_0^2 + k \int_D M \mathcal{F} \bigg(\frac{\psi_0}{M}\bigg)\, \dq\right]\!\dx + \int_0^t\int_\O \vr \,\ff \cdot \vu \, \dx\,\dt',\nn
\ea
which then implies (\ref{energy}). Indeed, thanks to the form of the Newtonian stress tensor in \eqref{i3}, direct calculations yield that
$$
\SSS : \Grad \vu = \mu^S \left|\frac{\nabla \vu + \nabla^{\rm T} \vu}{2} - \frac{1}{3} (\Div_x \vu) \II \right|^2 +\mu^B |\Div_x \vu|^2.
$$
Hence, by the positive definiteness of the Rouse matrix $A=(A_{ij})_{1\leq i,j\leq K}$, and recalling that the smallest eigenvalue of $A$ is $A_0>0$, we deduce \eqref{energy}.
\end{remark}

\subsection{Assumptions and main results}

We shall suppose that both $\mu^S$ and $\mu^B$ are $C^1$ functions of the polymer number density $\eta$, and we adopt the following assumptions: there exist positive constants $c_j, \ j=1,\dots,5$, and an $\o\in\R$ such that
\be\label{mu-eta1}
c_{1} (1+\eta)^\o   \leq \mu^S(\eta) \leq c_2 (1+\eta)^\o, \quad  |(\mu^S)'(\eta)| \leq c_3+c_4 (1+\eta)^{\o-1}, \quad 0   \leq \mu^B(\eta) \leq c_5 (1+\eta)^\o
\quad \forall\,\eta\geq 0.
\ee
In addition, for the sake of simplicity, we shall assume that
\be\label{press}
p(\vr) = a \vr^\gamma, \quad a > 0, \ \gamma > \frac32.
\ee

As the complete proof of the existence of dissipative weak solutions in the special case of constant viscosity coefficients is already very long and technical (cf. \cite{Barrett-Suli}), in the more general setting of polymer-number-dependent viscosity coefficients considered here we shall confine ourselves to establishing \emph{weak sequential stability} of the family of dissipative weak solutions, whose existence we shall assume; we shall however indicate in Section \ref{End} the main steps of a possible complete existence proof in the case of polymer-number-density-dependent viscosity coefficients.

Accordingly, the main result of the paper reads as follows.

\begin{theorem}[Weak Sequential Stability]\label{theorem} Let $\{(\vr_n,\vu_n,\psi_n,\eta_n)\}_{n\in \N}$ be a sequence of dissipative (finite-energy) weak solutions in the sense of Definition \ref{def-weaksl} associated with the initial data $\{(\vr_{0,n},\vu_{0,n},\psi_{0,n},\eta_{0,n})\}_{n\in\N}$ satisfying:
\ba\label{cond-ini-n}
&\vr_{0,n} \geq 0 \ \mbox{a.e. in} \ \O, \quad \vr_{0,n} \to \vr_0 \ \mbox{strongly in}\ L^\g(\O);\\
&\vu_{0,n} \to \vu_0 \ \mbox{in}\  L^r(\O;\R^3) \ \mbox{for some $r>1$}\ \mbox{such that}\ \vr_{0,n}|\vu_{0,n}|^2 \to \vr_0|\vu_0|^2 \ \mbox{strongly in}\  L^1(\O);\\
& \psi_{0,n} \geq 0 \ \mbox{a.e. in} \ \O\times D,\quad \psi_{0,n} \to \psi_{0},\quad \psi_{0,n} \left(\log \frac{\psi_{0,n}}{M}\right) \to \psi_0 \left(\log \frac{\psi_0}{M}\right) \ \mbox{strongly in}\  L^1(\O\times D);\\
&\eta_{0,n} = \int_{D}\psi_{0,n} \, \dq \to \eta_0 \ \mbox{strongly in} \ L^2 (\O).
\ea

Let $\vc{f} \in L^\infty((0,T) \times \Omega;\R^3)$.
Suppose that the exponent $\g$ in \eqref{press} and the parameter $\o$ in \eqref{mu-eta1} satisfy
\be\label{mu-eta2}
0\leq \o < \frac{5}{3}, \quad \g>\frac{3}{2}
\ee
or
\be\label{mu-eta3}
 -\frac{4}{3}<\o\leq 0, \quad \g>\frac{6}{4+3\o}.
\ee
Then, there exists a subsequence (not indicated) such that
$$
(\vr_n,\vu_n,\psi_n,\eta_n) \to (\vr,\vu,\psi,\eta) \quad \mbox{as $n\to \infty$, in the sense of distributions (at least weakly in}\ L^1),
$$
where the limit $(\vr,\vu,\psi,\eta)$ is a dissipative (finite-energy) weak solution in the sense of Definition \ref{def-weaksl} associated with the initial data $(\vr_0,\vu_0,\psi_0,\eta_0)$.

\end{theorem}

Before embarking on the proof of Theorem \ref{theorem} two remarks are in order.

\begin{remark}\label{rem-thm1} The strong convergence assumptions in \eqref{cond-ini-n} imply that
\ba\label{cond-ini-n1}
\vr_{0,n} \vu_{0,n} \to \vr_{0} \vu_{0}  \quad \mbox{strongly in}\ L^{\frac{2\g}{1+\g}}(\O;\R^3),\qquad \eta_0 = \int_D \psi_0\, \dq  \quad \mbox{a.e. in}\ \O.
\ea

\end{remark}

\begin{remark}\label{rem-thm2} It is important to note that we allow the viscosity coefficients $\mu^B$ and $\mu^S$ to decay to zero as the polymer number density tends to infinity; this is achieved at the expense of assuming a larger adiabatic exponent $\g$; cf. \eqref{mu-eta3}.

\end{remark}

The bulk of the rest of the paper is devoted to the proof of Theorem \ref{theorem}. Comments on the possibility of carrying out a complete proof of the existence
of dissipative (finite-energy) weak solutions are given in Section \ref{End}.  Throughout the rest of the paper, if there is no specification, $c$ will denote a positive constant depending only on the length $T$ of the time interval and the following quantity associated with the initial data:
$$
\sup_{n\in \N}\int_{\Omega} \bigg[
\frac{1}{2} \vr_{0,n} |\vu_{0,n} |^2 +  P(\vr_{0,n}) + \de\, \eta_{0,n}^2 + k \int_D M \mathcal{F} \bigg(\frac{\psi_{0,n}}{M}\bigg)\, \dq\bigg] \dx.
$$
We emphasize, however, that the value of $c$ may vary from line to line.

\section{Preliminaries}

In this section, we recall some concepts that will be used systematically throughout the rest of the paper, including Maxwellian-weighted Lebesgue and Sobolev spaces, embeddings of spaces of Banach-space-valued weakly-continuous functions, the Div-Curl lemma, and Riesz operators.

\subsection{Maxwellian-weighted spaces}\label{sec:def-LMr-HM1}
For any $r \in [1,\infty)$, $L_M^r(D)$ denotes the Maxwellian-weighted Lebesgue space over $D$ with norm
\be\label{def-LMD1}
\|u\|_{L^r_M(D)}:=\left(\int_{ D} M |u(x)|^r\, \dq \right)^{\frac{1}{r}}.
\ee
Similarly, we define $L_M^r(\O\times D):=L^r(\O;L^r_M(D))$ and the Maxwellian-weighted Sobolev spaces
\ba\label{def-SMD2}
H^1_M(D)&:=\left\{ u \in L_{\rm loc}^1( D) :  \|u\|_{H^1_M( D)}^2:=\int_{ D} M\big( |u|^2 + |\nabla_q u|^2 \big)\, \dq<\infty \right\}.\\
H^1_M(\O\times D)&:=\left\{ u \in L_{\rm loc}^1(\O\times D) :  \|u\|_{H^1_M(\O\times D)}^2:=\int_{\O\times D} M\big( |u|^2 + |\nabla_x u|^2 +|\nabla_q u|^2\big)\, \dq \ \dx<\infty \right\}.\nonumber
\ea
The proof of the following lemma can be found in Appendix C and Appendix D in \cite{Barrett-Suli2}.
\begin{lemma}\label{lem:HM1-emb}  The normed spaces $L^r_M (D)$, $L_M^r(\O\times D)$, $H^1_M(D)$ and $H^1_M(\O \times D)$ are Banach spaces.
The embedding $H^1(\O;L_M^2(D)) \hookrightarrow L^6(\O;L_M^2(D))$ is continuous, and the embeddings
$H^1_M(D) \hookrightarrow L_M^2(D)$, $H^1_M(\O\times D)  \hookrightarrow L^2_M(\O\times D)$ are compact.
\end{lemma}

\subsection{On $C_w([0, T];X)$ type spaces}

Let $X$ be a Banach space. We denote by $C_w([0, T];X)$ the set of all functions $u\in L^\infty(0, T;X)$ such that the mapping
$t \in  [0, T] \mapsto \langle \phi,u(t)\rangle_{X}\in \R$ is continuous on $[0, T]$ for all $\phi \in X'$. Here and throughout the paper, we use $X'$ to denote the
dual space of $X$, and $\langle \cdot,\cdot \rangle_{X}$ to denote the duality pairing between $X'$ and $X$.

Whenever $X$ has a predual $E$, in the sense that $E'=X$, we denote by $C_{w*}([0, T];X)$ the set of all functions $u\in L^\infty(0, T;X)$ such that the mapping $t \in  [0, T] \mapsto \langle u(t),\phi\rangle_{E}\in \R$ is continuous on $[0, T]$ for all $\phi \in E$. We reproduce Lemma 3.1 from \cite{Barrett-Suli}.
\begin{lemma}\label{lem-Cw-Cw*}
Suppose that $X$ and $Y$ are Banach spaces.
\begin{itemize}
\item[(i)] Assume that the space $X$ is reflexive and is continuously embedded in the space $Y$; then,
$$ L^\infty(0, T;X) \cap C_w([0, T];Y) = C_w([0, T];X).$$
\item[(ii)] Assume that $X$ has a separable predual $E$ and $Y$ has a predual $F$ such that $F$ is continuously
embedded in $E$; then,
   $$ L^\infty(0, T;X) \cap C_{w*}([0, T];Y) = C_{w*}([0, T];X).$$
\end{itemize}

\end{lemma}

We recall an Arzel\`{a}--Ascoli type result in $C_w([0,T];L^s(\O))$. We refer to Lemma 6.2 in \cite{N-book} for the proof.
\begin{lemma}\label{lem-Cw} Let $r, s \in (1,\infty)$ and let $\O$ be a bounded Lipschitz domain in $\R^d, \ d\geq 2$. Suppose  $\{g_n\}_{n\in \N}$ is a sequence of functions in $C_w([0, T]; L^s(\O))$ such that $\{g_n\}_{n\in \N}$ is bounded in $C([0, T]; W^{-1,r}(\O)) \cap \ L^\infty(0, T;L^s(\O))$. Then, there exists a subsequence (not indicated) such that the following hold:
\begin{itemize}
\item[(i)] $g_n\to g$ in $C_w([0, T]; L^s(\O))$;
\item[(ii)] If, in addition, $r\leq \frac{d}{d-1}$, or $r>\frac{d}{d-1}$ and $s>\frac{d\,r}{d+r}$, then $g_n\to g$ strongly in $C([0, T]; W^{-1,r}(\O))$.
\end{itemize}
\end{lemma}

\subsection{Div-Curl lemma}

We recall the celebrated Div-Curl lemma due to L. Tartar \cite{L.Tartar}.
\begin{lemma}\label{lem-div-curl2}
Let $Q\subset \R^d$ be a domain and let $\{({\bf U}_n,{\bf V}_n)\}_{n\in \N}$ be a sequence of functions such that
$$
{\bf U}_n \to {\bf U} \ \mbox{weakly in $L^p(Q;\R^d)$}, \quad {\bf V}_n \to {\bf V} \ \mbox{weakly in $L^{q}(Q;\R^d)$}, \quad \mbox{as $n\to \infty$},
$$
where
$
\frac{1}{p}+\frac 1q =\frac{1}{r}<1.
$
Suppose in addition that, for some $s>0$,
$$
\{\Div \,{\bf U}_n\}_{n\in \N} \ \mbox{is precompact in $W^{-1,s}(Q)$},\quad \{{\rm curl} \,{\bf V}_n\}_{n\in \N} \ \mbox{is precompact in $W^{-1,s}(Q;\R^{d\times d})$}.
$$
Then,
$
{\bf U}_n \cdot {\bf V}_n \to {\bf U} \cdot {\bf V} \ \mbox{weakly in}\ L^r(Q).
$

\end{lemma}

\subsection{On Riesz type operators}\label{sec:Riesz}

The Riesz operator $\RR_{j}$, $1\leq  j \leq d$, in $\R^d$ is defined as a Fourier integral operator with symbol $\frac{\xi_j}{|\xi|}$. That is, for any $u\in \mathcal{S}'(\R^d)$, where $\mathcal{S}'(\R^d)$ denotes the space of tempered distributions on $\R^d$, $\RR_j$ is defined by
$$
\RR_{j}[u] := \mathfrak{F}^{-1}\left[\frac{\xi_j}{|\xi|} \mathfrak{F}[u]\right],
$$
where $\mathfrak{F}$ is the Fourier transform and $\mathfrak{F}^{-1}$ is the inverse Fourier transform. We then define, for any $u\in \mathcal{S}'(\R^d)$,
$$
\RR_{ij} [u]:=\RR_i \circ \RR_j u= \mathfrak{F}^{-1}\left[\frac{\xi_i\xi_j}{|\xi|^2} \mathfrak{F}[u]\right],\quad 1\leq i,j\leq d.
$$

We define $\A_j$ by
$
\A_{j} [u] :=- \mathfrak{F}^{-1}\left[\frac{i\xi_j}{|\xi|^2} \mathfrak{F}[u]\right].
$
Since the derivative $\d_j$ and the Laplacian $\Delta$ can be seen as Fourier integral operators with symbols $i\xi_j$ and $-|\xi|^{-2}$, respectively, we can write
$$
\RR_{ij} = \d_i\d_j{\Delta}^{-1},\quad  \A_j = -\d_{j} \Delta^{-1}.
$$

Let the matrix-valued operator $\RR$ and the vector-valued operator $\A$ be defined by
\be\label{def-Reize}
\RR=(\RR_{ij})_{i,j=1}^d=(\nabla\otimes\nabla) \Delta^{-1},\quad \A:=(A_j)_{j=1}^d = - \nabla \Delta^{-1}.
\ee
We have
$
\RR_{ij}= - \d_i \A_j,\ \sum_{j=1}^d \RR_{jj}= -\sum_{j}\d_j\A_j =\II.
$ By Theorem 1.55 and Theorem 1.57 in \cite{N-book} the following result holds.
\begin{lemma}\label{lem-Riesz1} For any $p \in (1,\infty)$ the operators $\RR_j$ and $\RR_{ij}$, $1\leq i, j\leq d$, are bounded from $L^p(\R^d)$ to $L^p(\R^d)$. That is,
there exists a positive constant $c=c(p,d)$ such that
$$
\|\RR_j [u] \|_{L^p(\R^d)}+  \|\RR_{ij} [u] \|_{L^p(\R^d)} \leq c(p,d) \|u \|_{L^p(\R^d)}\qquad \forall\,u\in L^p(\R^d).
$$
Moreover, for any $u\in L^p(\R^d)$, $v\in L^{p'}(\R^d)$, $1<p<\infty$, we have
$$
\int_{\R^d} \RR_{ij}[u]\,  v \, \dx=\int_{\R^d} u \,\RR_{ij}[v]\, \dx.
$$
Further, for any $p \in (1,d)$, $\A_j$ is bounded from $L^p(\R^d)$ to $L^{\frac{dp}{d-p}}(\R^d)$; that is, there exists a positive constant $c=c(p,d)$ such that
$$
\|\A_j [u] \|_{L^{\frac{dp}{d-p}}(\R^d)} \leq c(p,d) \|u \|_{L^p(\R^d)}\qquad \forall\,u\in L^p(\R^d).
$$
\end{lemma}

The following commutator estimate is taken from Theorem 10.28 in \cite{F-N-book}, which is in the spirit of Coifman and Meyer \cite{CM}:
\begin{lemma}\label{lem-Riesz3}
Let $w\in W^{1,r}(\R^d)$ and $v\in L^p (\R^d)$ with
$
1<r<d,\ 1<p<\infty, \ \frac{1}{r}+\frac 1p-\frac 1d <1.
$ Then, for any $s>1$ satisfying
$
\frac{1}{r}+\frac 1p-\frac 1d < \frac 1s < \min\{1,\frac{1}{r}+\frac 1p\},
$
there exists a constant $c=c(r,p,s,d)$ such that:
$$
\| \RR_{ij}[w\,v] - w \RR_{ij}[v] \|_{W^{\b,s}(\R^d)}\leq c\, \|w\|_{W^{1,r}(\R^d)} \|v\|_{L^{p}(\R^d)},
$$
for any $i, j \in \{1, \dots, d\}$, where $\b\in (0,1)$ satisfies
$
\frac{\b}{d} =\frac 1d + \frac 1s - \frac{1}{r} - \frac 1p.
$

\end{lemma}

In Lemma \ref{lem-Riesz3}, we used the fractional-order Sobolev space $W^{\b,s}(\R^d)$. Let $\O$ be the whole space $\R^d$ or a bounded Lipschitz domain in $\R^d$. For any $\b \in (0,1)$ and $s\in [1,\infty) $, we define
\be\label{def-frac-Sob}
 W^{\b,s}(\O):=\left\{u\in L^s(\O) : \| u \|_{W^{\b,s}(\O)}:= \| u \|_{L^{s}(\O)} + \left(\int_\O\int_\O  \frac{|u(x)-u(y)|^s}{|x-y|^{d+\b s}} \ \dx \ {\rm d}y\right)^{\frac{1}{s}} <\infty\right\}.
\ee
We recall the following classical compact embedding theorem (see Theorem 7.1 in \cite{NPV}).
\begin{lemma}\label{lem-frac-sob} Let $\O\subset \R^d$ be a bounded Lipschitz domain and suppose that $\b \in (0,1)$ and $s\in [1,\infty)$; then, the embedding
of $W^{\b,s}(\O)$ into $L^s(\O)$ is compact, i.e. $W^{\b,s}(\O)\hookrightarrow \hookrightarrow L^s(\O)$.

\end{lemma}

\section{Uniform bounds}\label{sec:estimate}

Let $(\vr,\vu,\psi,\eta)$ be a dissipative (finite-energy) weak solution in the sense of Definition \ref{def-weaksl} with initial data $(\vr_0,\vu_0,\psi_0,\eta_0)$ satisfying \eqref{ini-data}. This section is devoted to establishing bounds on $(\vr,\vu,\psi,\eta)$ under the hypotheses \eqref{mu-eta1}--\eqref{mu-eta3}.

\subsection{Gronwall's inequality and uniform bounds}

We begin by noting that
\ba\label{ini-est2}
\int_{\Omega}  \int_D M \mathcal{F} \bigg(\frac{\psi_0}{M}\bigg)\,\dq \,\dx =\int_{\Omega}  \int_D  \bigg( \psi_0\log\bigg(\frac{\psi_0}{M}\bigg)-\psi_0 +M\bigg)\,\dq \,\dx.
\ea
As $ s\log s \geq s-1$ for all $s\geq 0$, it follows that $\psi_0\log\big(\frac{\psi_0}{M}\big)\geq \psi_0 - M$, which then implies that
$$
\left|\psi_0\log\bigg(\frac{\psi_0}{M}\bigg)-\psi_0 +M\right|\leq 2 \left|\psi_0\log\bigg(\frac{\psi_0}{M}\bigg)\right|.
$$
%
%
Thus, by \eqref{ini-data}, we have that
\be\label{ini-estf}
\int_\O \bigg[\frac{1}{2} \vr_{0} |\vu_{0} |^2 +  P(\vr_0) + \de\, \eta_0^2 + k \int_D M \mathcal{F} \bigg(\frac{\psi_0}{M}\bigg)\,\dq\bigg]\dx \leq c.
\ee

Since $\ff\in L^\infty((0,T)\times \O;\R^3)$ and
$$
|\vr \,\ff \cdot \vu |\leq |\ff| \, |\sqrt \vr|\,| \sqrt\vr \vu| \leq |\ff|\left( \vr + \vr|\vu|^2 \right) \leq |\ff|\left(1 + \vr^\g + \vr|\vu|^2 \right),
$$
by using Gronwall's inequality we deduce from \eqref{energy} and \eqref{ini-estf} that, for a.e. $t\in (0,T)$,
%
%
%
\ba\label{energy-f}
&\int_{\Omega} \bigg[ \frac{1}{2} \vr |\vu|^2 +  P(\vr) +\de\, \eta^2 + k \int_D M \mathcal{F} (\widetilde \psi)\, \dq\bigg] (t, \cdot) \,\dx\\
&\quad+ \int_0^t \intO{ \mu^S(\eta) \left|\frac{\nabla \vu + \nabla^{\rm T} \vu}{2} - \frac{1}{3} (\Div_x \vu) \II \right|^2 +\mu^B(\eta) |\Div_x \vu|^2  } \, \dt' + 2\, \e\, \de \int_0^t \intO{   |\nabla_x \eta|^2} \, \dt' \\
&\quad+\e\, k \int_0^t \intO{ \int_D M\left|\nabla_x \sqrt{\widetilde \psi}\, \right|^2\dq} \, \dt+\frac{k\,A_0}{4\lambda} \int_0^t \intO{ \int_D M \left|\nabla_{q}  \sqrt{\widetilde \psi}\,\right|^2\dq} \, \dt' \\
&\leq c(1+T)\, {\rm e}^{ct}.
\ea
In the rest of this section, we shall establish additional bounds on the unknowns, one by one, by using \eqref{energy-f}.

\subsection{Bounds on the fluid density and the polymer number density}

From \eqref{energy-f},  we have
\be\label{est-r-eta1}
\vr\in L^\infty(0,T;L^\gamma(\O)), \quad \eta\in L^\infty(0,T;L^2(\O))\cap L^2(0,T;W^{1,2}(\O)),  \quad \vr |\vu|^2 \in L^{\infty}(0,T;L^{1}(\O)).
\ee
By Sobolev embedding and interpolation it follows that
\be\label{est-r-eta3}
 \eta\in L^\infty(0,T;L^2(\O)) \cap L^2(0,T;L^6(\O)) \hookrightarrow L^{a}(0,T;L^{\frac{6a}{3a-4}}(\O)),\quad 2\leq a\leq \infty.
 \ee
The bounds in \eqref{est-r-eta1} then imply that
\be\label{est-ru2}
\vr \vu = \sqrt{\vr} \sqrt{\vr}\vu \in L^{\infty}(0,T;L^{\frac{2\g}{\g+1}}(\O;\R^3)).
\ee

\subsection{Bounds on the fluid velocity field}\label{sec:est-u}

By \eqref{energy-f} we deduce that
\be\label{est-u1}
g:= \sqrt{\mu^S(\eta)} \left|\frac{\nabla \vu + \nabla^{\rm T} \vu}{2} - \frac{1}{3} (\Div_x \vu) \II \right| \in L^2(0,T;L^2(\O)).
\ee
We recall that $\mu^S(\eta)$ fulfills the hypotheses stated in \eqref{mu-eta1} with $\o$ satisfying \eqref{mu-eta2} or \eqref{mu-eta3}. We begin by considering the case when $\o\geq 0$; \eqref{mu-eta1} then implies that $\mu^{-1}$ is uniformly bounded. Thus,
\be\label{est-u2}
 \left|\frac{\nabla \vu + \nabla^{\rm T} \vu}{2} - \frac{1}{3} (\Div_x \vu) \II \right| \in L^2(0,T;L^2(\O));
\ee
whence
Korn's inequality (see \cite{Dain}) with the no-slip boundary condition on $\vu$ implies that
\be\label{est-u4}
|\nabla_x \vu| \in L^2(0,T;L^2(\O)), \qquad \mbox{and hence}\qquad \vu \in L^2(0,T;W^{1,2}_0(\O;\R^3)).
\ee
On the other hand, when $\o\leq 0$ satisfies the constraint \eqref{mu-eta3}, that is $-4/3<\o\leq 0$, then, by \eqref{mu-eta1} and \eqref{est-r-eta3}, we have that
\be\label{est-mu1}
\mu(\eta)^{-1}\in L^\infty(0,T;L^{\frac{2}{|\o|}}(\O)) \cap L^{\frac{2}{|\o|}}(0,T;L^{\frac{6}{|\o|}}(\O)) \cap L^{\frac{10}{3|\o|}}((0,T)\times \O).
\ee
Thus, from \eqref{est-u1}, we deduce that
\be\label{est-u11}
 \left|\frac{\nabla \vu + \nabla^{\rm T} \vu}{2} - \frac{1}{3} (\Div_x \vu) \II \right| \in L^2(0,T;L^{\frac{4}{2+|\o|}}(\O)) \cap L^{\frac{4}{2+|\o|}}(0,T;L^{\frac{12}{6+|\o|}}(\O))\cap L^{\frac{20}{10+3|\o|}}((0,T)\times \O).
\ee
Hence, taking advantage of the no-slip boundary condition, we may use
Korn's inequality to obtain
\be\label{est-u13}
\vu \in L^2(0,T;W_0^{1,\frac{4}{2+|\o|}}(\O;\R^3)) \cap L^{\frac{4}{2+|\o|}}(0,T;W_0^{1,\frac{12}{6+|\o|}}(\O;\R^3))\cap L^{\frac{20}{10+3|\o|}}(0,T;W_0^{1,\frac{20}{10+3|\o|}}(\O;\R^3)).
\ee

\subsection{Bounds on the Newtonian stress tensor}

Consider first the case when \eqref{mu-eta2} holds with $\o\geq 0$. Let us write
$$
\mu^S(\eta) \left|\frac{\nabla \vu + \nabla^{\rm T} \vu}{2} - \frac{1}{3} (\Div_x \vu) \II \right| = g \sqrt{\mu^S(\eta)},
$$
where $g\in L^2((0,T)\times \O)$ is defined by \eqref{est-u1}.  Thanks to \eqref{mu-eta1}, \eqref{est-r-eta3}, \eqref{est-u4}, and by a similar argument as in the derivation of \eqref{est-u13}, we obtain
\be\label{est-SS1}
\SSS\in L^2(0,T;L^{\frac{4}{2+|\o|}}(\O;\R^{3\times 3})) \cap L^{\frac{4}{2+|\o|}}(0,T;L^{\frac{12}{6+|\o|}}(\O;\R^{3\times 3}))\cap L^{\frac{20}{10+3|\o|}}((0,T)\times \O;\R^{3\times 3}).
\ee
On the other hand, when $\o\leq 0$, by observing that $\sqrt{\mu^S} + \sqrt{\mu^B}\leq c<\infty$ we deduce that
\be\label{est-SS2}
\SSS\in L^2((0,T)\times\O;\R^{3\times 3}).
\ee

\subsection{Bounds on the probability density function}\label{sec:est-psi}

It follows from \eqref{energy-f} that
\be\label{est-psi1}
\mathcal{F}(\tpsi) \in L^\infty(0,T;L^1_M(\O\times D)), \quad  \sqrt{\tpsi} \in L^2(0,T;H_M^1(\O\times D)),
\ee
where we recall that $\tpsi := \psi/M$. Consequently, by Sobolev embedding and thanks to Lemma \ref{lem:HM1-emb} we then have that $\sqrt{\tpsi} \in L^2(0,T;L^6(\O;L^2_M(D)))$.

If $s>{\rm e}^2$, then $s\log s > 2s >2(s-1)$, which then implies that $\mathcal{F}(s)=s\log s -(s-1)>\frac{1}{2}{s\log s}$ for $s>{\rm e}^2$. Thus,
\ba\label{est-psi3}
\|\tpsi\log\tpsi\|_{L^1_M(\O\times D))}&=\int_{\O\times D} |\tpsi\log\tpsi|\,M\,\dq\,\dx = \int_{0 \leq \tpsi \leq {\rm e}^2} |\tpsi\log\tpsi|\,M\,\dq\,\dx +\int_{\tpsi > e^2} |\tpsi\log\tpsi|\,M\,\dq\,\dx \\
&  \leq  2 {\rm e}^2\int_{\O\times D}\,M\,\dq\,\dx + 2 \int_{\tpsi > {\rm e}^2} |\mathcal{F}(\tpsi)|\, M\,\dq\,\dx \leq c,
\ea
where we have used \eqref{pt1-Mi-Ui} and \eqref{est-psi1}. This implies that
\be\label{est-psi5}
\tpsi\log\tpsi \in L^\infty(0,T;L^1_M(\O\times D)),\quad \psi\log\psi \in L^\infty(0,T;L^1(\O\times D)).
\ee

Clearly,
\be\label{nabla-psi}
\nabla_{x}\tpsi = 2 \sqrt{\tpsi} \ \nabla_{x}\sqrt{\tpsi},\quad \nabla_{q}\tpsi = 2 \sqrt{\tpsi} \ \nabla_{q}\sqrt{\tpsi}.
\ee
By \eqref{nabla-psi} and H\"older's inequality we therefore have that
\ba\label{est-psi8}
\|\nabla_{q,x} \tpsi\|_{L^1_M(D;\R^{3(K+1)})} & = \int_{D} | M\,\nabla_{q,x}\tpsi | \, \dq = 2\int_{D} M \sqrt{\tpsi} \, \left|\nabla_{q,x}\sqrt{\tpsi}\,\right|\dq \\
&\leq 2 \left(\int_{D} M \left|\nabla_{q,x}\sqrt{\tpsi}\,\right|^2 \, \dq \right)^{\frac 12}\left(\int_{D} M \tpsi \, \dq \right)^{\frac 12}=2 \left\|\nabla_{q,x}\sqrt{\tpsi}\,\right\|_{L^2_M(D;\R^{3(K+1)})} \eta^{\frac 12}.
\ea
Thus, we benefit from the estimates in \eqref{est-r-eta3} for $\eta$ and deduce that
\ba\label{est-psi9}
&\nabla_{q,x} \tpsi \in L^2(0,T;L^{\frac 43}(\O;L^1_M(D;\R^{3(K+1)})))\cap  L^{\frac 43}(0,T;L^\frac{12}{7}(\O;L^1_M(D;\R^{3(K+1)}))).
\ea
By observing that $\|\nabla_x \tpsi\|_{L^1_M(D;\R^{3})}=\|\nabla_x\psi\|_{L^1(D;\R^{3})}$ we obtain
\ba\label{est-psi10}
\nabla_x \psi \in L^2(0,T;L^{\frac 43}(\O;L^1(D;\R^{3})))\cap  L^{\frac 43}(0,T;L^\frac{12}{7}(\O;L^1(D;\R^{3}))).
\ea

\subsection{Bounds on the extra-stress tensor}

We recall that the extra-stress tensor can be expressed as (see \eqref{weak-est})
$$\TT := k \left(\sum_{i=1}^K \CC_i(\psi)\right)-\left(k(K+1) \eta  + \de\,  \eta^2\right) \II. $$
Thanks to the bounds on $\eta$ stated in \eqref{est-r-eta1} and \eqref{est-r-eta3}, we have that
\be\label{est-TT1}
\eta^2\in L^\infty(0,T;L^1(\O)) \cap L^1(0,T;L^3(\O)).
\ee

According to \eqref{pt1-Mi-Ui}, we have $M=0$ on $\d D$. We then deduce by using \eqref{def-Ci} and \eqref{pt2-M} that
\ba\label{est-TT2}
\CC_i(\psi) &= \CC_i(M\tpsi) = \int_D M \tpsi\, U_i'\bigg(\frac{q_i^2}{2}\bigg)\, q_i q_i^{\rm T} \, \dq = - \int_{D} \tpsi\, (\nabla_{q_i} M)q_i^{\rm T}\, \dq\\
&=\int_{D} M\, (\nabla_{q_i}\tpsi) q_i^{\rm T}\, \dq + \left(\int_{D} M  \tpsi \, \dq\right)\II = \int_{D} M (\nabla_{q_i}\tpsi) q_i^{\rm T}\, \dq + \eta \II.
\ea
Thus, by \eqref{nabla-psi}, H\"older's inequality implies that
\ba\label{est-TT4}
\left|\int_{D} M (\nabla_{q_i}\tpsi) q_i^{\rm T}\, \dq \right|  \leq c \left(\int_{D} M \left(\nabla_{q_i}\sqrt{\tpsi}\right)^2 \, \dq \right)^{\frac 12}\left(\int_{D} M \tpsi \, \dq \right)^{\frac 12}= c \left(\int_{D} M \left(\nabla_{q_i}\sqrt{\tpsi}\right)^2 \, \dq \right)^{\frac 12} \eta ^{\frac 12}.\nn
\ea
By \eqref{est-psi1} we have
$$
\left(\int_{D} M \left(\nabla_{q_i}\sqrt{\tpsi}\right)^2 \dq \right)^{\frac 12} \in L^2(0,T;L^2(\O)).
$$
Thus, by \eqref{est-r-eta3} and \eqref{est-TT1}, we obtain
\be\label{est-TTf}
\TT \in L^2(0,T;L^{\frac 43}(\O;\R^{3\times 3}))\cap  L^{\frac 43}(0,T;L^\frac{12}{7}(\O;\R^{3\times 3})).
\ee

\subsection{Higher integrability of the fluid density and the pressure}\label{sec:est-higher-vr}

From the energy inequality \eqref{energy} we deduce that
$\vr\in L^\infty(0,T;L^\g(\O))$, which implies that $p(\vr)\in L^\infty(0,T;L^1(\O))$; unfortunately, $p(\vr)$ is only in $L^1$ with respect to the spatial variable, $x$. In order to improve the
integrability of the pressure with respect to the spatial variable, one may use the so-called Bogovski{\u{\i}} operator (see \cite{bog}), exactly as in \cite{FNP, F-book, N-book}.
We recall the following lemma whose proof can be found in Chapter III of Galdi's book \cite{Galdi-book}.
\begin{lemma}\label{lem-bog}
 Let $p \in (1,\infty)$ and suppose that $\O\subset \R^d$ is a bounded Lipschitz domain. Let $L^{p}_0(\Omega)$ be the space of $L^p(\Omega)$ functions with zero mean-value.
Then, there exists a linear operator $\mathcal{B}_\O$ from $L_0^p(\O)$ to $W_0^{1,p}(\O;\R^d)$ such that
$$
\Div_x\, \mathcal{B}_\O (f) =f \quad \mbox{in} \ \O; \quad \|\mathcal{B}_\O (f)\|_{W_0^{1,p}(\O;\R^d)} \leq c (d,p,\O)\, \|f\|_{L^p(\O)}\quad\, \forall\,  f\in L_0^p(\O),
$$
where the constant $c$ depends only on $p$, $d$ and the Lipschitz character of $\O$. If, in addition, $f=\Div_x \vg$ for some $\vg \in L^{q}, \  1<q<\infty, \ \vg \cdot {\bf n} = 0 $ on $\d\O$, then
$$
\|\mathcal{B}_\O (f)\|_{L^{q}(\O;\R^d)} \leq c (d, q,\O)\, \|\vg \|_{L^q(\O;\R^d)}.
$$

\end{lemma}

The key idea in establishing higher integrability of $\vr$ and $p(\vr)$ is to choose the following test function in \eqref{weak-form3}:
\be\label{bog-test}
\varphi (t,x) : =  \phi(t)\, \mathcal{B}_{\O} \bigg(S_\e\big[ b_{n}(\vr) \big] - \frac{1}{|\O|}\int_\O S_\e\big[ b_{n}(\vr)\big]\, \dx \bigg),
\ee
where $\phi\in C^\infty_c((0,T))$ is a nonnegative test function, $S_\e$ is the classical (Friedrichs) mollifier with respect to the spatial variable, and $\{b_n(\vr)\}_{n\in \N}$ is an increasing sequence of $C^1$ functions satisfying \eqref{cond-b-renormal}, which approximates the function $\vr^\th$. As in Lemma 2.1 in \cite{Feireisl-2001} we have, for any $b$ satisfying \eqref{cond-b-renormal},
\be\label{weak-renormal1}
\d_t b(\vr) +\Div_x (b(\vr)\vu) + (b'(\vr)\vr-b(\vr))\,\Div_x \vu =0 \ \mbox{in} \ \mathcal{D}' ((0,T)\times \R^3),
\ee
where the functions $\vr$ and $\vu$ are extended by zero outside $\O$. For any $n$, $b_n(\vr)$ satisfies \eqref{weak-renormal1}. After tedious but rather straightforward calculations (see \cite{Feireisl-2001,FNP} for details), one obtains, for $\th>0$ sufficiently small, that
\be\label{est-bog-p1}
\int_0^t\int_\O \phi(t)\, \vr^\g\, b_{n}(\vr)\, \dx \, \dt \leq c\qquad \forall\,n \geq 0.
\ee
Letting $n\to \infty$ in \eqref{est-bog-p1} and approximating $1$ by $C_c^\infty ((0,T))$ functions finally gives
\be\label{est-bog-pf1}
 \vr \in L^{\g+\th}((0,T)\times \O),\quad p(\vr) \in L^{1+\frac{\th}{\g}}((0,T)\times \O) \quad \mbox{for some $\th>0$}.
\ee

\medskip

\begin{remark}\label{rem-bog-th}

In order to obtain \eqref{est-bog-pf1} for some positive $\th$, the conditions imposed in \eqref{mu-eta2} and \eqref{mu-eta3} can be relaxed. By careful analysis the following constraints are found to be sufficient:
\be\label{mu-eta20}
\g>\frac 32,\ 0\leq \o < \frac{10}{3}  \quad \mbox{or} \quad -2<\o\leq 0,\  \g>\frac{6}{4+3\o}.
\ee
The more restrictive conditions featuring in \eqref{mu-eta2} and \eqref{mu-eta3} are needed later on, in Section \ref{vis-flux}, in order to prove the so-called {\rm effective viscous flux} equality (see Remark \ref{rem-mu-eta}).
\end{remark}

\subsection{Bounds on the time derivative and continuity}\label{sec:est-time}

This section is devoted to establishing bounds on the time derivatives $(\d_t \vr, \ \d_t\eta, \  \d_t (\vr \vu), \ \d_t \psi)$.

As $\d_t \vr = - \Div_x(\vr\vu)$, the bound in \eqref{est-ru2} implies that
$\d_t \vr \in L^\infty(0,T; W^{-1,\frac{2\g}{\g+1}}(\O)).$ Then, by \eqref{est-r-eta1} and Lemma \ref{lem-Cw-Cw*}, we have that
$$
\vr \in C_w([0,T]; L^\g(\O)).
$$

\medskip

Recall that $\d_t \eta = - \Div_x(\eta\vu) + \e \Delta_x \eta$.  For $\o\geq 0$ satisfying \eqref{mu-eta2}, we have $\vu \in L^2(0,T; W_0^{1,2}(\O;\R^3))$, which is embedded in $L^2(0,T; L^6(\O;\R^3))$. Thanks to \eqref{est-r-eta1}, we then obtain $\d_t \eta \in L^2(0,T; W^{-1,\frac{3}{2}}(\O)).$
If on the other hand $\o\leq 0$ satisfies \eqref{mu-eta3}, then we have $\vu \in L^2(0,T;W_0^{1,\frac{4}{2+|\o|}}(\O;\R^3)) \hookrightarrow L^2(0,T; L^\frac{12}{3|\o|+2}(\O;\R^3))$. Consequently, $\d_t \eta \in L^2(0,T; W^{-1,\frac{12}{3|\o|+8}}(\O))$.
Thus, in both cases, by \eqref{est-r-eta1} and Lemma \ref{lem-Cw-Cw*}, we have
$$
\eta \in C_w([0,T]; L^2(\O)).
$$

\medskip

Next, we recall that $\d_t(\vr\vu)=- \Div_x (\vr \vu \otimes \vu) - \nabla_x p(\vr) +
\Div_x \SSS +\Div_x \TT +\vr\, \ff$. We first consider the case with $\o\geq 0$ satisfying \eqref{mu-eta2}. By the estimates in Section \ref{sec:est-u} we have $\vu \in L^2(0,T; L^6(\O;\R^3))$. Since $\vr\in L^\infty(0,T;L^\g(\O))$  with $\g>\frac{3}{2}$, together with \eqref{est-r-eta1}, we obtain
\be\label{est-ruu}
\vr \vu \otimes \vu \in L^1(0,T; L^{\frac{3\g}{\g+3}}(\O;\R^{3\times 3}))\cap L^\infty(0,T; L^{1}(\O;\R^{3\times 3})) \hookrightarrow L^r(0,T; L^{r}(\O;\R^{3\times 3})) \quad \mbox{for some $r>1$}.
\ee
By \eqref{est-SS1}, \eqref{est-bog-pf1}, \eqref{est-TTf}, together with \eqref{est-ruu}, we deduce that
\be\label{est-dt-ru}
\d_t (\vr\vu) \in  L^r(0,T; W^{-1,r}(\O;\R^{3})) \quad \mbox{for some $r>1$}.
\ee
For the case with $\o\leq 0$ satisfying \eqref{mu-eta3}, a similar argument gives the same result as in \eqref{est-dt-ru}. Then, by \eqref{est-ru2}, \eqref{est-dt-ru} and Lemma \ref{lem-Cw-Cw*}, we have in both cases that
$$\vr\vu \in C_w([0,T]; L^\frac{2\g}{\g+1}(\O;\R^3)).$$

By the Fokker--Planck equation \eqref{eq-psi} for $\psi$, the estimates in \eqref{est-psi9} and \eqref{est-psi10}, direct calculations yield that
\be\label{est-dt-psi}
\d_t \psi \in
L^2(0,T; W^{s,2}(\O\times D)') \quad \mbox{with $s>1 + \frac{3}{2}(K+1)$},
\ee
where we have used that $W^{s,2}(\Omega \times D) \hookrightarrow W^{1,\infty}(\Omega \times D)$ for $s>1+\frac{3}{2}(K+1)$.
Moreover, by the estimates \eqref{est-psi5} and \eqref{est-dt-psi}, using Lemma \ref{lem-Cw-Cw*} and the same argument as in Section 4.5 in \cite{Barrett-Suli}, we deduce that  $$\psi \in C_w([0,T];L^1(\O\times D)).$$

\subsection{Summary}

We summarize the results obtained in this section. Let $(\vr,\vu,\psi,\eta)$ be a dissipative weak solution in the sense of Definition \ref{def-weaksl} associated
with initial data $(\vr_0,\vu_0,\psi_0,\eta_0)$ satisfying \eqref{ini-data}. Then,
\ba\label{est-sum1}
& \vr\in C_w([0,T);L^\gamma(\O))\cap L^{\g+\th}((0,T)\times \O);\quad  \eta \in C_w([0,T];L^2(\O)) \cap L^2(0,T;W^{1,2}(\O));\\
& \vr |\vu|^2 \in L^{\infty}(0,T;L^{1}(\O)),\quad \vr \vu \in C_w([0,T];L^{\frac{2\g}{\g+1}}(\O;\R^3));\quad  \TT \in L^2(0,T;L^{\frac 43}(\O;\R^{3\times 3}))\cap  L^{\frac 43}(0,T;L^\frac{12}{7}(\O;\R^{3\times 3}));\\
&\vu \in L^2(0,T;W^{1,2}_0(\O;\R^3)) \ \mbox{if $\o\geq 0$ satisfies \eqref{mu-eta2}};\\
& \vu \in L^2(0,T;W_0^{1,\frac{4}{2+|\o|}}(\O;\R^{3})) \cap L^{\frac{4}{2+|\o|}}(0,T;W_0^{1,\frac{12}{6+|\o|}}(\O;\R^{3}))\cap L^{\frac{20}{10+3|\o|}}(0,T;W_0^{1,\frac{20}{10+3|\o|}}(\O;\R^{3})) \\
&\qquad \mbox{if $\o\leq 0$ satisfies \eqref{mu-eta3}};\\
&\mathcal{F}(\tpsi),  \tpsi\log\tpsi, \ \tpsi \in L^\infty(0,T;L^1_M(\O\times D)),\quad \sqrt{\tpsi} \in L^2(0,T;H_M^1(\O\times D)),\\
&\qquad \nabla_{q,x} \tpsi \in L^2(0,T;L^{\frac 43}(\O;L^1_M(D;\R^{3(K+1)})))\cap  L^{\frac 43}(0,T;L^\frac{12}{7}(\O;L^1_M(D;\R^{3(K+1)})));\\
&\mathcal{F}(\psi),\  \psi\log\psi \  \in L^\infty(0,T;L^1(\O\times D)),\quad  \psi \in C_w([0,T];L^1(\O\times D)),\\
 &\qquad \nabla_x \psi \in L^2(0,T;L^{\frac 43}(\O;L^1(D;\R^{3})))\cap  L^{\frac 43}(0,T;L^\frac{12}{7}(\O;L^1(D;\R^{3})));
\ea
and

\begin{equation}\label{est-sum2}
\begin{aligned}
&\d_t \vr \in L^\infty(0,T; W^{-1,\frac{2\g}{\g+1}}(\O)), \quad \d_t (\vr\vu) \in  L^r(0,T; W^{-1,r}(\O;\R^{3})) &&\quad \mbox{for some $r>1$};\\
& \d_t \eta \in L^2(0,T; W^{-1,\frac{3}{2}}(\O)) &&\quad \mbox{if $\o\geq 0$ satisfies \eqref{mu-eta2}};\\
&\d_t \eta \in L^2(0,T; W^{-1,\frac{12}{3|\o|+8}}(\O)) &&\quad \mbox{if $\o\leq 0$ satisfies \eqref{mu-eta3}};\\
&\d_t \psi \in  L^2(0,T; W^{s,2}(\O\times D)') &&\quad \mbox{with $s>1 + \frac{3}{2}(K+1)$}.
\end{aligned}
\end{equation}
It is important to note that all of the above inclusions are consequences of bounds that depend only on the initial energy, the final time $T$, and the structural constants in the hypotheses imposed on the constitutive relations.

\section{Passing to the limit}\label{sec:lim}

Let $(\vr_n,\vu_n,\psi_n,\eta_n)_{n\in \N}$ be a sequence of the dissipative (finite-energy) weak solutions satisfying the assumptions in Theorem \ref{theorem}. Then, the energy inequality \eqref{energy} and \eqref{energy-f} give, respectively,
\ba\label{energy-n}
&\int_{\Omega} \bigg[
\frac{1}{2} \vr_n |\vu_n|^2 +  P(\vr_n) +\de\, \eta_n^2 + k \int_D M \mathcal{F} (\psi_n)\, \dq\bigg] (t, \cdot) \, \dx + 2\, \e\, \de \int_0^t \intO{   |\nabla_x \eta_n|^2} \, \dt'\\
&\quad+ \int_0^t \intO{ \mu^S(\eta_n) \left|\frac{\nabla \vu_n + \nabla^{\rm T} \vu_n}{2} - \frac{1}{3} (\Div_x \vu_n) \II \right|^2 +\mu^B(\eta_n) |\Div_x \vu_n|^2  } \, \dt' \\
&\quad+ \e\, k \int_0^t \intO{ \int_D M\left|\nabla_x \sqrt{\tpsi_n} \right|^2 \dq} \, \dt+\frac{k\,A_0}{4\lambda} \int_0^t \intO{ \int_D M \left|\nabla_{q} \sqrt{\tpsi_n}\right|^2 \dq} \, \dt'\\
&\leq \int_{\Omega} \bigg[
\frac{1}{2} \vr_{0,n} |\vu_{0,n} |^2 +  P(\vr_{0,n}) + \de\, \eta_{0,n}^2 + k \int_D M \mathcal{F} \bigg(\frac{\psi_{0,n}}{M}\bigg)\,\dq\bigg]\dx + \int_0^t\int_\O \vr_n \,\ff \cdot \vu_n \, \dx\, \dt'
\ea
 and, therefore,
 \ba\label{energy-f-n}
&\int_{\Omega} \bigg[ \frac{1}{2} \vr_n |\vu_n|^2 +  P(\vr_n) +\de\, \eta_n^2 + k \int_D M \mathcal{F} (\widetilde \psi_n)\,\dq\bigg] (t, \cdot) \, \dx + 2\, \e\, \de \int_0^t \intO{   |\nabla_x \eta_n|^2} \, \dt' \\
&\quad+ \int_0^t \intO{ \mu^S(\eta_n) \left|\frac{\nabla \vu_n + \nabla^{\rm T} \vu_n}{2} - \frac{1}{3} (\Div_x \vu_n) \II \right|^2 +\mu^B(\eta_n) |\Div_x \vu_n|^2  } \ \dt' \\
&\quad+\e\, k \int_0^t \intO{ \int_D M\left|\nabla_x \sqrt{\widetilde \psi_n} \right|^2\dq} \ \dt'+\frac{k\,A_0}{4\lambda} \int_0^t \intO{ \int_D M \left|\nabla_{q}  \sqrt{\widetilde \psi_n}\right|^2\dq} \, \dt' \\
&\leq c(1+T)\, {\rm e}^t.
\ea
Thus, from the results in Section \ref{sec:estimate}, this solution sequence $(\vr_n,\vu_n,\psi_n,\eta_n)_{n\in \N}$ satisfies the uniform bounds  \eqref{est-sum1} and \eqref{est-sum2}.

The present section is devoted to studying the limit of this solution sequence. We remark that, throughout this section, the limits are taken up to subtractions of subsequences without identification.

\subsection{Convergence of the fluid density}\label{sec:vr-lim}

We have that $\{\vr_n\}_{n\in \N}$ is a sequence in $C_w([0,T];L^\gamma(\O))$ satisfying
\be\label{vrn-est}
\sup_{n\in \N} \left(\|\vr_n\|_{L^\infty(0,T;L^\gamma(\O))} +  \|\d_t \vr_n \|_{L^\infty(0,T; W^{-1,\frac{2\g}{\g+1}}(\O))}  \right) \leq c.
\ee
By Sobolev embedding one has
$$
L^\g(\O) \hookrightarrow W^{-1,\frac{3\g}{3-\g}}(\Omega) \quad \mbox{ if $\g<3$};\quad L^\g(\O) \hookrightarrow W^{-1,s}(\Omega) \quad \mbox{ for any $s \in (1,\infty)$, \ if $\g\geq3$}.
$$
Thus, by applying Lemma \ref{lem-Cw}, we deduce that
\ba\label{vrn-lim}
& \vr_n \to \vr \quad \mbox{in}\  C_w([0,T];L^\gamma(\O));\\
& \vr_n \to \vr \quad \mbox{strongly in} \  C([0,T];W^{-1,r}(\O)) \quad \mbox{for any $r\geq \frac{3}{2}$ such that $\frac{3r}{3+r}<\g$}.
\ea

\subsection{Convergence of the fluid velocity field}\label{sec:lim-u}

By Sobolev embedding, for the case $\o\geq 0$, we have
\be\label{vun-limit1}
\vu_n \to \vu \quad \mbox{weakly in}\  L^2(0,T;W^{1,2}_0(\O;\R^3))\  \mbox{and in}\  L^2(0,T;L^{6}(\O;\R^3)),
\ee
while for the case $\o\leq 0$, we have
\ba\label{vun-limit2}
&\vu_n \to \vu \quad \mbox{weakly in}\  L^2(0,T;W_0^{1,\frac{4}{2+|\o|}}(\O;\R^3))\  \mbox{and weakly in}\  L^2(0,T;L^{\frac{12}{2+3|\o|}}(\O;\R^3)),\\
&\vu_n \to \vu \quad \mbox{weakly in}\  L^{\frac{4}{2+|\o|}}(0,T;W_0^{1,\frac{12}{6+|\o|}}(\O;\R^3))\  \mbox{and weakly in}\  L^{\frac{4}{2+|\o|}}(0,T;L^{\frac{12}{2+|\o|}}(\O;\R^3)),\\
&\vu_n \to \vu \quad \mbox{weakly in}\  L^{\frac{20}{10+3|\o|}}(0,T;W_0^{1,\frac{20}{10+3|\o|}}(\O;\R^3)) \  \mbox{and weakly in}\  L^{\frac{20}{10+3|\o|}}(0,T;L^{\frac{60}{10+9|\o|}}(\O;\R^3)).
\ea

\subsection{Convergence of the nonlinear terms $\vr_n\vu_n$ and $\vr_n \vu_n\otimes \vu_n$}\label{sec:lim-convective}

We first consider the case with $\o\geq 0$ satisfying \eqref{mu-eta2}. Since $\g>\frac 32>\frac 65$, we have, because of \eqref{vrn-lim}, that
\ba\label{vrn-lim1}
 \vr_n \to \vr \quad \mbox{strongly in} \  C([0,T];W^{-1,2}(\O)).
\ea
Together with \eqref{vun-limit1}, we have
\be\label{vrvun-lim0}
\vr_n\vu_n \to \vr\vu \quad \mbox{in $\mathcal{D}'((0,T)\times \O;\R^{3})$.}
\ee
By observing the uniform estimate
$$
\sup_{n \in \mathbb{N}}\left(\|\vr_n\vu_n\|_{L^\infty(0,T;L^{\frac{2\g}{1+\g}}(\O;\R^{3}))} + \|\vr_n\vu_n\|_{L^2(0,T;L^{\frac{6\g}{6+\g}}(\O;\R^{3}))}\right)\leq c
$$
we have
\be\label{vrvun-lim1}
\vr_n\vu_n \to \vr\vu \quad \mbox{weakly* in} \ L^\infty(0,T;L^{\frac{2\g}{1+\g}}(\O;\R^{3})) \ \mbox{and weakly in} \ L^2(0,T;L^{\frac{6\g}{6+\g}}(\O;\R^{3})).
\ee

Moreover, by \eqref{est-sum1} and \eqref{est-sum2}, we have $\vr_n\vu_n \in C_w([0,T];L^{\frac{2\g}{\g+1}}(\O;\R^d))$ and $\{\d_t (\vr_n\vu_n)\}_{n\in \N}$ is uniformly bounded in $L^r(0,T;W^{-1,r}(\O;\R^3))$ for some $r>1$. By Lemma \ref{lem-Cw} and the interpolation argument in Section \ref{sec:vr-lim} we have
\ba\label{vrvun-lim2}
&\vr_n\vu_n \to \vr\vu \quad \mbox{in}\  C_w([0,T];L^{\frac{2\g}{1+\g}}(\O;\R^3));\\
&\vr_n\vu_n \to \vr\vu \quad \mbox{strongly in} \  C([0,T];W^{-1,r}(\O;\R^3)) \quad \mbox{for any $r\geq \frac{3}{2}$ such that $\frac{3r}{3+r}<\frac{2\g}{1+\g}$}.
\ea

The fact that $\g>\frac 32$ implies $\frac{2\g}{1+\g} > \frac{6}{5}$. This gives $\vr_n\vu_n \to \vr\vu$ strongly in $C([0,T];W^{-1,2}(\O;\R^3))$. Together with \eqref{vun-limit1}, we have for the sequence of convective terms:
\be\label{vrvuvun-lim0}
\vr_n\vu_n\otimes \vu_n \to \vr\vu\otimes \vu \quad \mbox{ in $\mathcal{D}'((0,T)\times \O;\R^{3\times 3})$.}
\ee

The case $\o\leq 0$ is dealt with similarly to the case $\o\geq 0$, and we obtain the results stated in \eqref{vrvun-lim0} and \eqref{vrvuvun-lim0}, so we omit the details. We merely remark that when  $\o\leq 0$, then we have weaker integrability for $\vu$; this is, however, compensated by supposing stronger integrability for $\vr$ through hypothesis \eqref{mu-eta3}.

 We also remark that the hypotheses stated in \eqref{mu-eta2} and \eqref{mu-eta3} can be relaxed in this part of the analysis: the constraints stated in \eqref{mu-eta20} in Remark \ref{rem-bog-th} are sufficient.

\subsection{Convergence of the polymer number density}\label{sec:lim-eta}

We begin by noting that the sequence $\{\eta_n\}_{n\in \N}$ is contained in $C_w([0,T];L^2(\O))$ and satisfies the bound
\be\label{etan-est1}
\sup_{n \in \mathbb{N}} \left(\|\eta_n\|_{L^\infty(0,T;L^2(\O)) \cap L^2(0,T;W^{1,2}(\O))} + \|\d_t \eta_n\|_{L^2(0,T;W^{-1,\frac{12}{3|\o|+8}}(\O))}\right) \leq c.
\ee
Thus, by Lemma \ref{lem-Cw},
\be\label{etan-lim1}
\eta_n \to  \eta \quad \mbox{strongly in}\ C_w([0,T];L^2(\O)) \ \mbox{and weakly in}\ L^2(0,T;W^{1,2}(\O)).
\ee
Thanks to the compact Sobolev embedding $W^{1,2}(\O) \hookrightarrow \hookrightarrow L^q(\O)$, $q<6$, the
Aubin--Lions--Simon compactness theorem (see \cite{Lions-InC} or \cite{Simon}) implies that
\be\label{etan-lim2}
\eta_n \to  \eta \quad \mbox{strongly in}\  L^2(0,T;L^q(\O))\quad \mbox{for any $q<6$}.
\ee
By interpolation we also have
\be\label{etan-lim3}
\eta_n \to  \eta \quad \mbox{strongly in}\  L^q((0,T) \times \O)\quad \mbox{for any $q<\frac{10}{3}$}.
\ee

We still need to show that $\eta=\int_D \psi \ \dq$, where $\psi$ is the limit of $\psi_n$. This will be done later on, in Section  \ref{sec:limit-psi}.

\subsection{Convergence of the Newtonian stress tensor}\label{sec:lim-SS}
When $0\leq \o  <\frac{5}{3}$ as in \eqref{mu-eta2}, by \eqref{mu-eta1} and \eqref{etan-lim3}, we have that
\be\label{mu-lim1}
(\mu^S(\eta_n),\mu^B(\eta_n)) \to  (\mu^S(\eta),\mu^B(\eta))  \quad \mbox{strongly in}\  L^{\frac{10}{3\o}}((0,T)\times \O).
\ee
Together with \eqref{vun-limit1}, we deduce that
\be\label{Newtonian-lim1}
\SSS_n  \to \SSS :=\mu^S(\eta) \left( \frac{\nabla \vu + \nabla^{\rm T} \vu}{2} - \frac{1}{3} (\Div_x \vu) \II \right) + \mu^B(\eta) (\Div_x \vu) \II\quad {\rm weakly\  in} \ L^{\frac{10}{3\o+5}}((0,T)\times \O;\R^{3\times 3}).
\ee

On the other hand, when $-\frac{4}{3}<\o\leq 0$ as in \eqref{mu-eta2}, then we have
\be\label{mu-lim2}
(\mu^S(\eta_n),\mu^B(\eta_n)) \to  (\mu^S(\eta),\mu^B(\eta))  \quad \mbox{strongly in}\  L^q(0,T;L^q(\O)) \quad \mbox{for any $q<\infty$}.
\ee
Thus, by \eqref{vun-limit2}, we have, for any  $r<\frac{20}{3\o+10}$,
\ba\label{Newtonian-lim10}
\SSS_n  \to \SSS :=\mu^S(\eta) \left( \frac{\nabla \vu + \nabla^{\rm T} \vu}{2} - \frac{1}{3} (\Div_x \vu) \II \right) + \mu^B(\eta) (\Div_x \vu) \II\quad
{\rm weakly~in} \ L^{r}((0,T)\times \O;\R^{3\times 3}).
\ea

We remark that, by rewriting
\ba\label{Newtonian-est2}
\SSS_n  = \sqrt{\mu^S(\eta_n)} \sqrt{\mu^S(\eta_n)} \left( \frac{\nabla \vu_n + \nabla^{\rm T} \vu_n}{2} - \frac{1}{3} (\Div_x \vu_n) \II \right) + \sqrt{\mu^B(\eta_n)}\sqrt{\mu^B(\eta_n)} (\Div_x \vu_n) \II,
\ea
we can relax the constraints in \eqref{mu-eta2} and \eqref{mu-eta3} as in \eqref{mu-eta20} in Remark \ref{rem-bog-th}.

\subsection{Convergence of the probability density function}\label{sec:limit-psi}

First of all, by the energy inequality \eqref{energy-f-n} we have
\be\label{est-psin1}
\sup_{n\in\N}\left( \left\|\mathcal{F}(\tpsi_n)\right\|_{L^\infty(0,T;L^1_M(\O\times D))} + \left\|\sqrt{\tpsi_n}\,\right\|_{L^2(0,T;H_M^1(\O\times D))}  \right)\leq c.
\ee

As in Section 4 in \cite{Barrett-Suli} and Section 5 in \cite{Barrett-Suli4}, we use Dubinski{\u\i}'s compactness theorem (cf. \cite{Dubin}; see also \cite{Barrett-Suli3},  Theorem 3.1 or \cite{Barrett-Suli}) by setting
\ba\label{dubinskii-spaces}
&X:= L^1_M(\O\times D),\quad X_0:=\{ \vp\in X : \vp\geq 0,\ \sqrt{\vp} \in H^1_M(\O\times D) \},\\
&X_1: = M^{-1} W^{s,2}(\O\times D)':=\{ M^{-1} \vp : \vp \in  W^{s,2}(\O\times D)'\}  \quad \mbox{with $s>1+\frac{3}{2}(K+1)$},
\ea
where $X_0$ is a seminomed space (in the sense of Dubinski{\u\i}) with seminorm defined by
$$
[\vp]_{X_0}: = \|\vp\|_{X} + \int_{\O\times D} M\bigg( \left|\nabla_x \sqrt{\vp}\,\right|^2 + \left|\nabla_q \sqrt{\vp}\,\right|^2 \bigg) \, \dq \, \dx.
$$

By \eqref{est-sum1} and \eqref{est-sum2} we have that
\be\label{psin-est1}
\mbox{$\{\widetilde \psi_n\}_{n\in \N}$ is uniformly bounded in $L^1(0,T;X_0)$\quad and \quad $\{\d_t \widetilde \psi_n\}_{n\in \N}$ is uniformly bounded in $L^2(0,T;X_1)$}.
\ee
The continuity of the embedding $X\hookrightarrow X_1$ and the compactness of the embedding $X_0 \hookrightarrow X$ are shown in Section 5 in \cite{Barrett-Suli4}. Then, by virtue of Dubinski{\u\i}'s compact embedding theorem, we have that
\be\label{psin-lim1}
\tpsi_n \to \tpsi \quad \mbox{strongly in} \ L^1(0,T;L_M^1(\Omega\times D)),
\ee
which is equivalent to
\be\label{psin-lim2}
\psi_n \to \psi \quad \mbox{strongly in} \ L^1(0,T;L^1(\Omega\times D)).
\ee
This implies that
\be\label{psin-lim3}
\eta_n = \int_D \psi_n \,\dq \to \int_D \psi \, \dq \quad \mbox{strongly in} \ L^1((0,T)\times \Omega);
\ee
in addition, by the uniqueness of the limit,  the function $\eta$ obtained in Section \ref{sec:lim-eta} satisfies $\eta = \int_D \psi \, \dq$.

Since $\mathcal{F}(\tpsi)$ is nonnegative on $(0,T)\times \Omega \times D$,  by applying Fatou's lemma we have that
\be\label{psin-est3}
 \left\|\mathcal{F}(\tpsi)\right\|_{L^\infty(0,T;L^1_M(\O\times D))} \leq \liminf_{n\to \infty} \left\|\mathcal{F}(\tpsi_n)\right\|_{L^\infty(0,T;L^1_M(\O\times D))} \leq c.
\ee
By the same technique as in Section \ref{sec:est-psi} we deduce from \eqref{psin-est3} that
\be\label{psin-est4}
\tpsi\log\tpsi, \ \tpsi \in L^\infty(0,T;L^1_M(\O\times D)),\quad \psi\log\psi, \ \psi \in L^\infty(0,T;L^1(\O\times D)).
\ee

By \eqref{psin-est1} we have
\be\label{psin-est5}
\d_t \psi \in L^2(0,T;W^{s,2}(\O\times D)') \quad \mbox{for any $s>1 + \frac{3}{2}(K+1)$}.
\ee

From the estimates in \eqref{psin-est4} and \eqref{psin-est5}, by using the second part of Lemma \ref{lem-Cw-Cw*} and the same argument as in Section 4.5 in \cite{Barrett-Suli}, we deduce that
$
 \psi \in C_w([0,T];L^1(\O\times D)).
$

Again, we write $\nabla_{q_i}\tpsi_n = 2 \sqrt{\tpsi_n}\nabla_{q_i}\sqrt{\tpsi_n}$.  By \eqref{psin-lim1} we have
$$
\sqrt{\tpsi_n} \to \sqrt{\tpsi} \quad  \mbox{strongly in} \ L^2(0,T;L_M^2(\Omega\times D)).
$$
Further, by \eqref{est-psin1}, we have that
$$
\nabla_q \sqrt{\tpsi_n} \to \nabla_q \sqrt{\tpsi} \quad \mbox{weakly in} \ L^2(0,T;L_M^2(\Omega\times D;\R^{3K})).
$$
Thus,
\be\label{psin-limf1}
\nabla_{q}\tpsi_n \to \nabla_{q}\tpsi \quad \mbox{weakly in} \ L^1(0,T;L_M^1(\Omega\times D;\R^{3K})).
\ee

Similarly, we also have that
\be\label{psin-limf2}
\nabla_{x}\tpsi_n \to \nabla_{x}\tpsi \quad \mbox{weakly in} \ L^1(0,T;L_M^1(\Omega\times D;\R^{3})),\quad
\nabla_{x}\psi_n \to \nabla_{x}\psi \quad \mbox{weakly in} \ L^1(0,T;L^1(\Omega\times D;\R^{3})).
\ee

\subsection{Convergence of the extra-stress tensor}\label{sec:lim-TT}

We recall the formula for the extra-stress tensor. By \eqref{weak-est} and \eqref{est-TT2} we have that
\be\label{TTn-est1}
\TT_n := k \left(\sum_{i=1}^K \int_{D} M (\nabla_{q_i}\tpsi_n) q_i^{\rm T}\, \dq \right) - \left(k \,\eta_n  + \de\,  \eta_n^2\right) \II,
 \ee
which, because of \eqref{est-TTf}, satisfies
\be\label{TTn-est2}
\sup_{n\in\N}\left(\|\TT_n\|_{L^2(0,T;L^{\frac 43}(\O;\R^{3\times 3}))} + \|\TT_n\|_{L^{\frac 43}(0,T;L^\frac{12}{7}(\O;\R^{3\times 3}))}\right)<c.
\ee
Thus,
\be\label{TTn-lim1}
\TT_n\to \bar \TT \quad \mbox{weakly in} \ L^2(0,T;L^{\frac 43}(\O;\R^{3\times 3})) \cap L^{\frac 43}(0,T;L^\frac{12}{7}(\O;\R^{3\times 3})).
\ee

We still need to show
\be\label{TTn-lim2}
\bar\TT = \TT := k \left(\sum_{i=1}^K \int_{D} M (\nabla_{q_i}\tpsi) q_i^{\rm T}\, \dq \right) - \left(k \,\eta  + \de\,  \eta^2\right) \II.
\ee
Indeed, using the same argument as in Section 4.5 in \cite{Barrett-Suli}, it follows that
\be\label{TTn-lim-strf}
\TT_n \to \TT \quad \mbox{strongly in}\  L^1((0,T)\times \O;\R^{3\times 3}).
\ee
By the uniform estimates \eqref{TTn-est2} and interpolation we also have that
\be\label{TTn-lim-strff}
\TT_n\to \TT \quad \mbox{strongly in } L^r((0,T)\times \O;\R^{3\times 3}) \quad \mbox{for any $r<\frac{20}{13}$}.
\ee

\subsection{Convergence of the nonlinear terms in the Fokker--Panck equation}\label{sec:lim-nonl-FP}

In this section, we study the limits associated with the nonlinear terms in the Fokker--Planck equation \eqref{eq-psi}, which are $\Div_x(\vu_n \psi_n)$ and  $\Div_{q_i}\!\left((\nabla_x \vu_n )\, q_i\, \psi_n\right)$. We would like to show that
\be\label{nonlin-FP-lim0}
\vu_n \psi_n \to  \vu \psi \quad \mbox{in $\mathcal{D}'((0,T)\times \O \times D;\R^{3})$}\quad\mbox{and}\quad \left(\nabla_x \vu_n \right) q_i\psi_n \to \left(\nabla_x \vu \right) q_i \psi \quad \mbox{in}\ \mathcal{D}'((0,T)\times \O\times D;\R^{3\times 3}).
\ee

To this end, let $\vvarphi \in C_c^\infty((0,T)\times \O\times D;\R^3)$ be a test function. We then have that
\ba\label{nonlin-FP-est1}
&\int_0^T\int_{\O\times D}(\vu_n \psi_n - \vu \psi) \cdot  \vvarphi \, \dq\, \dx \, \dt \\
&= \int_0^T\int_{\O\times D}(\vu_n -\vu)\psi \cdot \vvarphi \, \dq\, \dx \, \dt  +  \int_0^T\int_{\O\times D}\vu_n(\psi_n-\psi) \cdot \vvarphi \, \dq\, \dx \, \dt =:I_1^n+I_2^n.
\ea

For $I_1^n$, we have
\ba\label{nonlin-FP-est2}
I_1^n:= \int_0^T\int_{\O\times D}(\vu_n -\vu)\psi \cdot \vvarphi \, \dq\, \dx \, \dt  = \int_0^T\int_{\O}(\vu_n -\vu) \cdot\left(\int_D \psi \, \vvarphi \, \dq \right) \dx \, \dt.
\ea
We observe that
\be\label{nonlin-FP-est3}
\left|\int_D \psi \, \vvarphi \, \dq \right| \leq c\,\int_D \psi\,  \dq = c\, \eta \in L^\infty(0,T;L^2(\O))\cap L^2(0,T;L^6(\O)).
\ee
We recall the weak convergence results for $\vu_n$ stated in \eqref{vun-limit1} and \eqref{vun-limit2}; in particular we have that $\vu_n \to \vu$ weakly in $L^2(0,T;L^{\frac{12}{2+3|\o|}}(\O;\R^3))$, with
$$
\frac{2+3|\o|}{12} + \frac{1}{6} = \frac{1}{3} + \frac{|\o|}{4} <1,
$$
under hypothesis \eqref{mu-eta2} or hypothesis \eqref{mu-eta3}. Therefore,
$\lim_{n\to \infty} I_1^n = 0.$

\medskip

For $I_2^n$, we have
\ba\label{nonlin-FP-est4}
I_2^n:= \int_0^T\int_{\O\times D}\vu_n (\psi_n-\psi) \cdot \vvarphi \, \dq\, \dx \, \dt = \int_0^T\int_{\O} \vu_n \cdot \left(\int_D (\psi_n-\psi) \, \vvarphi \, \dq \right) \dx \,\dt.
\ea
By the strong convergence of $\psi_n$ stated in \eqref{psin-lim2} we have that
\be\label{nonlin-FP-lim2}
\int_D (\psi_n-\psi) \, \vvarphi \, \dq \to \mathbf{0}\quad \mbox{a.e. in $(0,T)\times \O$}.
\ee
Similarly to \eqref{nonlin-FP-est4}, we have that
\be\label{nonlin-FP-est5}
\left|\int_D (\psi_n-\psi) \, \vvarphi \, \dq\right| \in L^\infty(0,T;L^2(\O))\cap L^2(0,T;L^6(\O)) \hookrightarrow L^{\frac{10}{3}}((0,T)\times \O).
\ee
Thanks to \eqref{nonlin-FP-lim2} and \eqref{nonlin-FP-est5}, Vitali's theorem implies that
\be\label{nonlin-FP-lim3}
\int_D (\psi_n-\psi) \, \vvarphi \, \dq  \to 0 \quad \mbox{in}\  L^{q}((0,T)\times \O;\R^3) \quad \mbox{for any $q<\frac{10}{3}$}.
\ee
It follows from \eqref{vun-limit1} (for $\omega \geq 0$) and from \eqref{vun-limit2} (for $\omega \leq 0$) that
$\lim_{n\to \infty} I_2^n = 0.$ Therefore,
\be\label{nonlin-FP-limf1}
\vu_n \psi_n \to  \vu \psi \quad \mbox{in}\ \mathcal{D}'((0,T)\times \O\times D;\R^3).
\ee

\medskip

To prove the convergence of $(\nabla_x \vu_n)\,q_i \psi_n$ to $(\nabla_x \vu )\,q_i\psi$ in $\mathcal{D}'((0,T)\times \O\times D;\R^3)$ one can proceed similarly. We point out that, when $\o\leq 0$, the requirement that $\o > -\frac{4}{3}$, appearing in \eqref{mu-eta3}, is really needed. Indeed, for any test function $\vvarphi$, just as in \eqref{nonlin-FP-lim3}, we have that
\be\label{nonlin-FP-lim5}
\int_D (\psi_n-\psi) \, \vvarphi \,q_i \ \dq  \to \mathbf{0} \quad \mbox{in}\  L^q(0,T;L^q(\O;\R^3)) \quad \mbox{for any $q<\frac{10}{3}$}.
\ee
According to the bounds established in Section \ref{sec:lim-u} for the case $\o\leq 0$, the sequence $\left\{\nabla_x \vu_n\right\}_{n\in \N}$ is uniformly bounded in $L^{\frac{20}{10+3|\o|}}((0,T)\times\O;\R^{3\times 3})$. To deduce the desired convergence result, we therefore need that
\be\label{mu-eta30}
\frac{10+3|\o|}{20} + \frac{3}{10}<1  \Longleftrightarrow |\o| < \frac{4}{3}.
\ee

\subsection{Convergence of the fluid pressure}
As was shown in Section \ref{sec:est-higher-vr} concerning higher integrability of the fluid density and pressure, we have that
\be\label{pn-lim1}
p(\vr_n) \to \overline{p(\vr)} \quad \mbox{weakly in} \ L^{1+\frac{\th}{\g}}((0,T)\times \O).
\ee

By applying this, together with the convergence results from the previous sections, in particular the convergence results for the nonlinear terms stated in Section \ref{sec:lim-convective}, Section \ref{sec:lim-SS} and Section \ref{sec:lim-TT}, we deduce that
\be\label{i1-0}
\d_t \vr + \Div_x (\vr \vu) = 0 \quad \mbox{in}\  \mathcal{D}'((0,T)\times \O),
\ee
\be\label{i2-0}
 \d_t (\vr\vu)+ \Div_x (\vr \vu \otimes \vu) +\nabla_x \overline{p(\vr)} -
\Div_x \SSS =\Div_x \TT +\vr\, \ff\quad \mbox{in}\  \mathcal{D}'((0,T)\times \O;\R^3),
\ee
where the Newtonian stress tensor $\SSS$  and the extra-stress tensor $\TT$ are defined by \eqref{Newtonian-lim1} and  \eqref{TTn-lim2}, respectively.

Concerning the convergence of the nonlinear terms in the Navier--Stokes--Fokker--Planck system with variable viscosity coefficients considered here, it remains to show that $p(\vr) = \overline{p(\vr)},$
which, because of the strict convexity of $p(\cdot)$, is equivalent to the strong convergence of $\vr_n$:
\be\label{pn-limf2}
\vr_n \to \vr \quad \mbox{a.e. in} \ (0,T)\times \O.
\ee

This is also one of the main difficulties in the study of global existence of weak solutions to the compressible Navier--Stokes equations (see \cite{Lions-C, FNP,F-book,N-book}), where the so called \emph{effective viscous flux} introduced by P. L. Lions \cite{Lions-C} plays a crucial role. It turns out that the effective viscous flux as a whole is more regular than its components.
  We will prove the following lemma, which is in the spirit of Proposition 5.1 in \cite{Feireisl-2001}.
\begin{lemma}\label{lem-viscous-flux}
For any $\phi\in \mathcal{D}(0,T)$, $\vp \in \mathcal{D}(\O)$, we have that
\ba\label{vis-flux}\lim_{n\to \infty}\int_0^T \intO{\phi(t)\vp(x) \left[p(\vr_n) - \left(\frac{2\mu^S(\eta_n)}{3}-\mu^B(\eta_n)\right) \Div_x \vu_n \right] T_k(\vr_n) }\, \dt\\
=\int_0^T \intO{\phi(t)\vp(x) \left[ \overline{p(\vr)} - \left(\frac{2\mu^S(\eta)}{3}-\mu^B(\eta)\right) \Div_x \vu\right] \overline{T_k(\vr)} }\, \dt,
\ea
where $T_k$ is a cut-off function defined by $T_k(\cdot) :=  k \,T(\frac{\cdot}{k})$ for some concave function $T\in C^\infty([0,\infty))$ such that $T(s)=s$ for $0\leq s\leq 1$ and $T(s)=2$ for $s\geq 3$. Here $\overline{T_k(\vr)}$ is the weak* limit of the sequence $\{T_k(\vr_n)\}_{n\in \N}$ in $L^\infty((0,T)\times \O)$ as $n$ goes to infinity.
\end{lemma}

The viscosity coefficients $\mu^B$ and $\mu^S$ in our case are not constant, which gives rise to additional complications in the proof of this lemma. We shall employ the commutator estimates stated in Lemma \ref{lem-Riesz3} and the techniques used in Section 3.6.5 in \cite{F-N-book} to overcome the resulting difficulties.

\begin{proof}[Proof of Lemma \ref{lem-viscous-flux}] As in the proof of the effective viscous flux lemma for the compressible Navier--Stokes equations with constant viscosity coefficients (see Proposition 5.1 in \cite{Feireisl-2001} or Proposition 7.36 in \cite{N-book}), we introduce the following test functions:
\be\label{def-test-flux}
\vv_n(t,x) = \phi(t) \vp(x) \A \left[T_k(\vr_n) \right],\quad \vv(t,x) = \phi(t) \vp(x) \A \left[ \overline{T_k(\vr)}\right],
\ee
where $\phi\in C_c^\infty(0,T), \ \vp \in C_c^\infty(\O)$,  and $\A$ is the Fourier integral operator introduced in \eqref{def-Reize}. We remark that $\vr_n$ and $\vr$ are extended by zero outside of $\O$ in \eqref{def-test-flux}.

Taking these $\vv_n$ and $\vv$ as test functions in the weak formulation \eqref{weak-form3} and the weak formulation of \eqref{i2-0}, respectively, results in
\begin{align}
\label{weak-vn}
&\int_0^T \intO{ \big[ \vr_n \vu_n \cdot \partial_t \vv_n + (\vr_n \vu_n \otimes \vu_n) : \Grad \vv_n  + p(\vr_n)\, \Div_x \vv_n - \SSS_n : \Grad \vv_n \big]} \,\dt
\nonumber\\
&\hspace{3in}= \int_0^T \intO{ \TT_n : \nabla_x \vv_n  - \vr_n\, \ff \cdot \vv_n} \, \dt,\\
\label{weak-v}
&\int_0^T \intO{ \left[ \vr \vu \cdot \partial_t \vv + (\vr \vu \otimes \vu) : \Grad \vv  + \overline{p(\vr)}\, \Div_x \vv - \SSS : \Grad \vv \right]}\,\dt
\nonumber\\
&\hspace{3in}= \int_0^T \intO{ \TT : \nabla_x \vv  - \vr\, \ff \cdot \vv} \, \dt.
\end{align}

The main idea is to pass to the limit $n\to \infty$ in \eqref{weak-vn}, and compare the resulting limit with \eqref{weak-v}, which will ultimately imply our desired result \eqref{vis-flux}. Since, following the contributions of Lions \cite{Lions-C} and Feireisl \cite{F-book}, this type of argument is by now well understood, instead of including all of the details here we shall focus on the terms that do not appear in the case of the compressible Navier--Stokes equations with constant viscosity coefficients. These `new' terms are the following:
\ba\label{flux-new-terms1}
&\int_0^T \intO{ \TT_n : \nabla_x \vv_n} \, \dt,\quad \int_0^T \intO{ \TT : \nabla_x \vv } \, \dt,\quad \int_0^T \intO{  \SSS_n : \Grad \vv_n } \, \dt,\quad \int_0^T \intO{  \SSS : \Grad \vv } \, \dt,\nn
\ea
and the goal is to prove that
\ba\label{flux-new-terms2}
\lim_{n\infty}\int_0^T \intO{ \TT_n : \nabla_x \vv_n} \, \dt = \int_0^T \intO{ \TT : \nabla_x \vv } \, \dt,
\ea
and
\ba\label{flux-new-terms3}
&\lim_{n\to\infty}\int_0^T \intO{  \SSS_n : \Grad \vv_n } \, \dt -  \int_0^T \intO{  \SSS : \Grad \vv } \, \dt \\
&= \lim_{n\to\infty} \int_0^T \intO{\phi(t)\vp(x) \left[ \left(\frac{2\mu^S(\eta_n)}{3}-\mu^B(\eta_n)\right) \Div_x \vu_n \right] T_k(\vr_n) }\, \dt \\
&\quad - \int_0^T \intO{\phi(t)\vp(x) \left[ \left(\frac{2\mu^S(\eta)}{3}-\mu^B(\eta)\right) \Div_x \vu\right] \overline{T_k(\vr)} }\, \dt.
\ea

By the strong convergence of $\TT_n$ stated in \eqref{TTn-lim-strff} in Section \ref{sec:lim-TT} it is straightforward to show the convergence result \eqref{flux-new-terms2}, so we only focus on proving \eqref{flux-new-terms3}. We begin by noting that
\ba\label{flux-new31}
\int_0^T \intO{  \SSS_n : \Grad \vv_n } \, \dt  = \int_0^T \phi(t)\intO{  \SSS_n :\, \left(\nabla_x \vp \otimes \A[ T_k(\vr_n)]\right)} \, \dt + \int_0^T \phi(t)\intO{ \vp \,  \SSS_n : \RR [ T_k(\vr_n)]} \, \dt,\nn
\ea
where $\RR$ is the Riesz operator defined in \eqref{def-Reize} in Section \ref{sec:Riesz}.

For any fixed $k\in\N$,  the sequence $\{ T_k(\vr_n)\}_{n\in \N}$ is uniformly bounded in $ L^\infty(0,T;L^r(\R^3))$ for all $r \in [1, \infty]$. Thus, by Lemma \ref{lem-Riesz1}, we have that
\be\label{A-Tk-vrn-1}
\sup_{n\in \N} \left\| \A[T_k(\vr_n)]\right\|_{L^\infty(0,T;W^{1,r}(\O;\R^3))} \leq c\quad \mbox{for any $r \in (1,\infty)$}.
\ee

Observing that $T_k(\cdot)$ fulfills the properties in \eqref{cond-b-renormal}, it follows from \eqref{weak-renormal1} that
\be\label{renomal-Tk}
\d_t T_k(\vr_n) + \Div_x\left(T_k(\vr_n)\vu_n\right) + \left(T_k'(\vr_n)\vr_n-T_k(\vr_n)\right)\Div_x \vu_n =0\quad \mbox{in $ \mathcal{D}'((0,T)\times \R^3)$}.
\ee
This implies that
\ba\label{A-Tk-vrn-2}
\d_t \A[ T_k(\vr_n)]= \A[\d_t T_k(\vr_n)] = - \A\left[ \Div_x\left(T_k(\vr_n)\vu_n\right) \right] - \A\left[ \left(T_k'(\vr_n)\vr_n-T_k(\vr_n)\right)\Div_x \vu_n \right].
\ea
Since
$
\A_j \Div_x = -\sum_{i=1}^3\d_j \d_i \Delta^{-1}  = - \sum_{i=1}^3\RR_{ij}
$
are Riesz type operators, by Lemma \ref{lem-Riesz1} we have that
$$
\|\A_j \left[ \Div_x \vvarphi \right]\|_{L^r(\R^3:\R^3)}\leq c \, \|\vvarphi\|_{L^r(\R^3;\R^3)} \quad \mbox{for any $\vvarphi\in L^r(\R^3;\R^3)$ and any $r \in (1,\infty)$}.
$$
By using \eqref{vun-limit1} and \eqref{vun-limit2} we deduce that
\ba\label{A-Tk-vrn-3}
\left\|\A\left[ \Div_x\left(T_k(\vr_n)\vu_n\right) \right]\right\|_{L^2(0,T;L^6(\O;\R^3))} &\leq c &&\quad \mbox{when $\o\geq 0$}; \\
\left\|\A\left[ \Div_x\left(T_k(\vr_n)\vu_n\right) \right]\right\|_{L^2(0,T;L^{\frac{12}{2+6|\o|}}(\O;\R^3))} &\leq c &&\quad\mbox{when $\o\leq 0$}.
\ea
Furthermore, by Lemma \ref{lem-Riesz1} in conjunction with \eqref{vun-limit1} and \eqref{vun-limit2} we deduce that
\ba\label{A-Tk-vrn-4}
\left\| \A\left[ \left(T_k'(\vr_n)\vr_n-T_k(\vr_n)\right)\Div_x \vu \right] \right\|_{L^2(0,T;L^6(\O;\R^3))} &\leq c &&\quad \mbox{when $\o\geq 0$}; \\
\left\| \A\left[ \left(T_k'(\vr_n)\vr_n-T_k(\vr_n)\right)\Div_x \vu \right] \right\|_{L^2(0,T;L^{\frac{12}{2+6|\o|}}(\O;\R^3))} &\leq c &&\quad \mbox{when $\o\leq 0$}.
\ea
Thus,
\ba\label{A-Tk-vrn-5}
\left\| \d_t \A[ T_k(\vr_n)]  \right\|_{L^2(0,T;L^6(\O;\R^3))} &\leq c &&\quad \mbox{when $\o\geq 0$};\\
\left\|  \d_t \A[ T_k(\vr_n)] \right\|_{L^2(0,T;L^{\frac{12}{2+6|\o|}}(\O;\R^3))} &\leq c &&\quad \mbox{when $\o\leq 0$}.
\ea
It follows from the uniform estimates in \eqref{A-Tk-vrn-1} and \eqref{A-Tk-vrn-5} and the Aubin--Lions--Simon compactness theorem that
\ba\label{A-Tk-vrn-f}
\A[T_k(\vr_n)] \to \overline{\A[T_k(\vr)]} = \A[\overline{T_k(\vr)}] \quad \mbox{strongly in $L^\infty(0,T;L^r(\O;\R^3))$ \quad for any $r\in (1,\infty)$}.
\ea
Together with the weak convergence of $\SSS_n$ stated in \eqref{Newtonian-lim1} and \eqref{Newtonian-lim10}, we deduce that
\ba\label{flux-new32}
 \lim_{n\to\infty}\int_0^T \phi(t)\intO{  \SSS_n : (\nabla_x \vp \otimes \A[ T_k(\vr_n)])} \, \dt = \int_0^T \phi(t)\intO{  \SSS : (\nabla_x \vp \otimes \A[ \overline{T_k(\vr)}])} \, \dt.
\ea
Therefore, to show \eqref{flux-new-terms3}, it suffices to prove that
\ba\label{flux-new33}
& \lim_{n\to \infty} \int_0^T \phi(t)\intO{ \vp \,  \SSS_n : \RR [ T_k(\vr_n)]} \, \dt -  \int_0^T \phi(t)\intO{ \vp \,  \SSS : \RR [ \overline{T_k(\vr)}]} \, \dt\\
 &= \lim_{n\to\infty} \int_0^T \intO{\phi(t)\vp(x) \left[ \left(\frac{2\mu^S(\eta_n)}{3}-\mu^B(\eta_n)\right) \Div_x \vu_n \right] T_k(\vr_n) }\, \dt \\
&\quad - \int_0^T \intO{\phi(t)\vp(x) \left[ \left(\frac{2\mu^S(\eta)}{3}-\mu^B(\eta)\right) \Div_x \vu\right] \overline{T_k(\vr)} }\, \dt.
\ea

By \eqref{i3} we have that
\ba\label{flux-new34}
&\int_0^T \phi(t)\intO{ \vp \,  \SSS_n : \RR [ T_k(\vr_n)]} \ \dt= \int_0^T \phi(t)\intO{ \vp \,  \mu^S(\eta_n) \left( \frac{\nabla \vu_n + \nabla^{\rm T} \vu_n}{2} \right) : \RR [ T_k(\vr_n)]} \, \dt\\
&\qquad + \int_0^T \phi(t)\intO{ \vp \,  \left(\mu^B(\eta_n)-\frac{\mu^S(\eta_n)}{3}\right) (\Div_x \vu_n) \II: \RR [ T_k(\vr_n)]} \, \dt\\
& = \int_0^T \phi(t)\intO{ \vp \,  \mu^S(\eta_n) \nabla \vu_n  : \RR [ T_k(\vr_n)]} \, \dt + \int_0^T \phi(t)\intO{ \vp \,  \left(\mu^B(\eta_n)-\frac{\mu^S(\eta_n)}{3}\right) (\Div_x \vu_n) \, [ T_k(\vr_n)]} \, \dt,
\ea
where have we used the fact that $\RR$ is symmetric (see Lemma \ref{lem-Riesz1}) and that $\II : \RR = \sum_{i=1}^3 \RR_{ii}=\II$.
Further, by Lemma \ref{lem-Riesz1}, and noting that $\sum_{i,j=1}^3 \RR_{ij}\partial_j \vu_n^i = \Div_x \vu_n$, we have that
\ba\label{flux-new35}
& \int_0^T \phi(t)\intO{ \vp \,  \mu^S(\eta_n) \nabla \vu_n  : \RR [ T_k(\vr_n)]} \, \dt = \int_0^T \phi(t)\intO{ \sum_{i,j=1}^3 \RR_{ij} \left[\vp \,  \mu^S(\eta_n)\, \d_j \vu_n^i \right] T_k(\vr_n)} \, \dt\\
&\quad =\int_0^T \phi(t)\intO{ \sum_{i,j=1}^3 \left( \vp \,  \mu^S(\eta_n)\, \RR_{ij} \left[ \d_j \vu_n^i \right] T_k(\vr_n) + R_n^{ij} \, T_k(\vr_n) \right)\!} \, \dt \\
&\quad = \int_0^T \phi(t)\intO{ \left(\vp \,  \mu^S(\eta_n)\, (\Div_x \vu_n)  \, T_k(\vr_n) + \sum_{i,j=1}^3 R_n^{ij} \, T_k(\vr_n) \right)\!} \, \dt,
\ea
where $R_n^{ij}$ is the commutator defined by $R_n^{ij}:= \RR_{ij} \left[\vp \,  \mu^S(\eta_n) \d_j \vu_n^i \right]- \vp \,  \mu^S(\eta_n) \RR_{ij} \left[ \d_j \vu_n^i \right].$ We will use Lemma \ref{lem-Riesz3} to derive uniform bounds on $R_n^{ij}$. We first consider the case $0\leq \o < \frac 53$. Then, by \eqref{mu-eta1} and \eqref{etan-est1}, the sequence
$\{\nabla_x \mu^S(\eta_n)\}_{n\in\N} = \{(\mu^S)'(\eta_n) \nabla_x \eta_n\}_{n\in\N}$ is uniformly bounded in $L^2(0,T;L^2(\O;\R^3))$ when $0\leq \o\leq 1$, and in $L^{2}(0,T;L^\frac{2}{\o}(\O;\R^3))$ when $1\leq \o < \frac 53$.
Thus,
$$
\sup_{n\in \N}\| \mu^S(\eta_n) \|_{L^2(0,T;W^{1,2}(\O))} \leq c \quad \mbox{when}\ 0\leq \o\leq 1;\qquad \sup_{n\in \N}\| \mu^S(\eta_n) \|_{L^2(0,T;W^{1,\frac{2}{\o}}(\O))} \leq c \quad \mbox{when}\ 1 \leq \o < \frac 53.
$$

By \eqref{vun-limit1} the sequence $\{\nabla_x \vu_n\}_{n\in\N}$ is uniformly bounded in $ L^2(0,T; L^{2}(\O;\R^{3 \times 3}))$. Furthermore, we note that
\ba\label{mu-eta10}
\frac{1}{2} + \frac{1}{2} -\frac{1}{3} = \frac{2}{3}< 1;\quad \frac{1}{2}+\frac{\o}{2} - \frac{1}{3} =\frac{1+3\o}{6} < 1, \quad \mbox{as long as}\ 0\leq \o < \frac{5}{3}.
\ea
Hence, by Lemma \ref{lem-Riesz3}, for any $s>1$ such that
\ba\label{est-mun-2}
 \frac{1}{s} >\frac{1}{2}+\frac{1}{2} - \frac{1}{3}=\frac{2}{3} \quad \mbox{when}\ 0\leq \o\leq 1;\quad \frac{1}{s} >\frac{1}{2}+\frac{\o}{2} - \frac{1}{3} =\frac{1+3\o}{6} \quad \mbox{when}\ 1\leq \o < \frac{5}{3},
\ea
we have, for $r_\o=2$ when $0\leq \o\leq 1$ and for $r_\o=\frac{2}{\o}$ when $1\leq \o < \frac{5}{3}$, that
\ba\label{flux-commu-2}
&\left\|R_n^{ij}\right\|_{W^{\b,s}(\R^3)} \leq c\, \| \vp \mu^S(\eta_n) \|_{W^{1,r_\o}(\O)} \|\d_j \vu_n^i\|_{L^2(\O)},
\ea
where $\b\in (0,1)$ is such that
\ba\label{est-mun-3}
\frac{\b}{3} =\frac 13 + \frac 1s - 1 \quad \mbox{when}\ 0\leq \o\leq 1;\quad \frac{\b}{3} =\frac 13 + \frac 1s - \frac{1+\o}{2}\quad \mbox{when}\ 1\leq \o < \frac 53.
\ea
Therefore, for some $s>1$ and some $0<\b<1$, we have the uniform estimate
\ba\label{flux-commu-3}
\sup_{n\in \N} \left\|R_n^{ij} \right\|_{L^1(0,T;W^{\b,s}(\R^3))} \leq c\, \sup_{n\in \N} \| \vp \mu^S(\eta_n) \|_{L^2(0,T;W^{1,r_\o}(\O))}\, \sup_{n\in \N} \|\d_j \vu_n^i\|_{L^2(0,T;L^2(\O))}\leq c.
\ea
Hence, by the strong convergence of $\eta_n$ shown in Section \ref{sec:lim-eta} we have that
\ba\label{flux-commu-4}
R_n^{ij} \to R^{ij}:= \RR_{ij} \left[\vp \,  \mu^S(\eta) \d_j \vu^i \right]- \vp \,  \mu^S(\eta) \RR_{ij} \left[ \d_j \vu^i \right] \quad \mbox{weakly in} \ L^{\frac{10}{3\o+5}}((0,T)\times \O).
\ea

By the boundedness of $T_k$ we have
\be\label{Tk-weak}
T_k(\vr_n)\to \overline{T_k(\vr)} \quad \mbox{weakly in $L^r((0,T)\times \O)$\quad for any $1<r<\infty$}.
\ee
Next, we want to prove the convergence of the product
\ba\label{flux-commu-5}
R_n^{ij}\, T_k(\vr_n) \to R^{ij} \overline{T_k(\vr)} \quad \mbox{weakly in} \ L^{r}(0,T;L^{r}(\O)) \quad \mbox{for any $r<\frac{10}{3\o+5}$}.
\ea
We shall use the Div-Curl lemma to this end in the time-space variables by setting
\be\label{div-curl-set}
{\bf U}_n := (T_k(\vr_n), T_k(\vr_n) \vu_n ),\quad {\bf V}_n := (R_n^{ij},0,0,0).
\ee
By \eqref{renomal-Tk} we then have that
\be\label{div-curl1}
\Div_{t,x} {\bf U}_n = \d_t T_k(\vr_n) + \Div_x\!\left(T_k(\vr_n)\vu_n\right) = - \left(T_k'(\vr_n)\vr_n-T_k(\vr_n)\right)\Div_x \vu_n,
\ee
which is uniformly bounded in $L^2((0,T)\times \O)$ and, therefore,
\[\{\Div_{t,x} {\bf U}_n\}_{n \in \N}\quad \mbox{is precompact in $W^{-1,2}((0,T)\times \O)$}.\]
Next, on observing that ${\rm curl}_{t,x}\,  {\bf V}_n$ does not include the time-derivative of $R_n^{ij}$, by \eqref{flux-commu-3} and Lemma \ref{lem-frac-sob} we have that
\[
\{{\rm curl}_{t,x} {\bf V}_n\}_{n\in \N} \quad \mbox{is precompact in $W^{-1,s}((0,T)\times \O;\R^{3\times 3})$}.
\]
Thus, the Div-Curl lemma (Lemma \ref{lem-div-curl2}) implies \eqref{flux-commu-5}. This gives
\ba\label{flux-new36}
\lim_{n\to \infty} \int_0^T \phi(t)\intO{ \sum_{i,j=1}^3  R_n^{ij} \, T_k(\vr_n) } \ \dt = \int_0^T \phi(t)\intO{ \sum_{i,j=1}^3  R^{ij} \, T_k(\vr) } \ \dt.
\ea

We shall now briefly summarize the proof of \eqref{flux-new36} when $-\frac{4}{3}<\o\leq 0$.   By \eqref{mu-eta1} and \eqref{etan-est1} the sequence $\{ \mu^S(\eta_n)\}_{n\in\N}$ is uniformly bounded in $L^2(0,T;W^{1,2}(\O))$ when $\o\leq 0$. By \eqref{vun-limit2} the sequence $\{\nabla_x \vu_n\}_{n\in\N}$ is uniformly bounded in $ L^2(0,T; L^{\frac{4}{2+|\o|}}(\O);\R^{3\times 3})$. We have that
\ba\label{mu-eta10-w<0}
\frac{1}{2}+\frac{2+|\o|}{4} - \frac{1}{3} < 1, \quad \mbox{as long as}\ -\frac{4}{3}< \o \leq 0.
\ea
Then, by Lemma \ref{lem-Riesz3}, for any $s>1$ such that
\ba\label{est-mun-2-w<0}
& \frac{1}{s} > \frac{1}{2}+\frac{2+|\o|}{4} - \frac{1}{3},
\ea
we have that
\ba\label{flux-commu-2-w<0}
&\left\|R_n^{ij}\right\|_{W^{\b,s}(\R^3)} \leq c\, \| \vp \mu^S(\eta_n) \|_{W^{1,2}(\O)} \|\d_j \vu_n^i\|_{L^\frac{4}{2+|\o|}(\O)},
\ea
where $\b\in (0,1)$ is such that
\ba\label{est-mun-3-w<0}
\frac{\b}{3} =\frac 13 + \frac 1s - \frac{1}{2}-\frac{2+|\o|}{4}.
\ea
Therefore, for some $s>1$ and some $\b>0$, the following uniform estimate holds:
\ba\label{flux-commu-3-w<0}
\sup_{n\in \N} \left\|R_n^{ij} \right\|_{L^1(0,T;W^{\b,s}(\R^3))} \leq c\, \sup_{n\in \N} \| \vp \mu^S(\eta_n) \|_{L^2(0,T;W^{1,2}(\O))}\,\sup_{n\in \N}  \|\d_j \vu_n^i\|_{L^2(0,T;L^\frac{4}{2+|\o|}(\O))}\leq c.
\ea
Thanks to the strong convergence of $\eta_n$ shown in Section \ref{sec:lim-eta}, we have that
\ba\label{flux-commu-4-w<0}
 R_n^{ij} \to R^{ij}:= \RR_{ij} \left[\vp \,  \mu^S(\eta) \d_j \vu^i \right]- \vp \,  \mu^S(\eta) \RR_{ij} \left[ \d_j \vu^i \right], \\
 \mbox{weakly in} \ L^{r}((0,T)\times \O),\ \mbox{for any $r<\frac{20}{10+3|\o|}$}.
\ea
Thus, the Div-Curl lemma implies that
\ba\label{flux-commu-5-w<0}
R_n^{ij}\, T_k(\vr_n) \to R^{ij} \overline{T_k(\vr_n)} \quad \mbox{weakly in} \ L^{r}(0,T;L^{r}(\O)) \quad \mbox{for any $r<\frac{20}{10+3|\o|}$},
\ea
which, once again, implies \eqref{flux-new36}.

\medskip

Finally, from \eqref{flux-new34}, \eqref{flux-new35} and \eqref{flux-new36}, we deduce the desired result \eqref{flux-new33}. This implies \eqref{flux-new-terms3} by using \eqref{flux-new32}. At the same time, by tedious but, by now, standard calculations, as in the proof of the effective viscous flux lemma in \cite{Lions-C, F-book, N-book}, we have that
\ba\label{flux-new-terms3f}
\lim_{n\to\infty}\int_0^T \intO{  \SSS_n : \Grad \vv_n } \ \dt -  \int_0^T \intO{  \SSS : \Grad \vv } \ \dt =0.
\ea
Combining \eqref{flux-new-terms3f} with \eqref{flux-new-terms3}, we obtain \eqref{vis-flux} and complete the proof of Lemma \ref{lem-viscous-flux}.
\end{proof}

\medskip

With Lemma \ref{lem-viscous-flux} in hand, the proof of the strong convergence result \eqref{pn-limf2} then proceeds along a well-understood route (see for example \cite{Feireisl-2001}, from Section 6 to Section 8), so we shall not dwell on the details here. In particular, as in Proposition 7.1 in \cite{Feireisl-2001}, the limit $(\vr, \vu)$ satisfies \eqref{weak-renormal1} in the sense of renormalized weak solutions.

\medskip

It remains to show that the limit $(\vr,\vu,\psi,\eta)$ satisfies the energy inequality \eqref{energy}. This is easily seen by noting the strong convergence assumption on the initial data in \eqref{cond-ini-n} and passage to the limit $n\to \infty$ in \eqref{energy-n}, which directly imply the energy
inequality \eqref{energy} 
by the application of Tonelli's sequential weak (weak*) lower semicontinuity theorem (cf. Theorem 10.16 in \cite{RR-book}, for example,) to the terms appearing on the left-hand side of \eqref{energy-n}; in particular, the sequential weak lower semicontinuity of the $L^p$ norm, $1< p<\infty$, the sequential weak* lower semicontinuty of the $L^\infty$ norm and inequality \eqref{psin-est3} are used.
Thus we have shown that the limit $(\vr,\vu,\psi,\eta)$ is a dissipative (finite-energy) weak solution in the sense of Definition \ref{def-weaksl}. That completes the proof of Theorem \ref{theorem}.

\begin{remark}\label{rem-mu-eta}
The constraint $\o<\frac{5}{3}$ for the case $\o\geq 0$ is crucially determined by \eqref{mu-eta10} and the condition $\frac{10}{3\o+5}>1$, which appears in \eqref{flux-commu-4}, while the constraint $\o>-\frac{4}{3}$ for the case $\o\leq 0$ is crucially determined by \eqref{mu-eta10-w<0}. It is unclear whether,
with our present techniques at least, these restrictions on $\o$ can be relaxed.
\end{remark}

\section{Conclusion: existence of dissipative weak solutions}\label{sec:end}

\label{End}

The conclusions of Theorem \ref{theorem} \emph{do not}, of course, imply the existence of dissipative (finite-energy) weak solutions to the Navier--Stokes--Fokker--Planck system with polymer-number-density-dependent viscosity coefficients. A rigorous proof of the existence of dissipative weak solutions would require the following:
\begin{itemize}

\item a suitable approximation scheme, compatible with the energy inequality and the compactness arguments presented in this paper;

\item a rigorous proof of the existence of a solution to the approximation scheme;

\item proof of the convergence of the sequence of approximate solutions to a dissipative weak solution, mimicking the arguments presented in this paper.

\end{itemize}

Given the formal similarity of the present model to the one studied in \cite{Barrett-Suli}, a natural approach would be to adjust the approximation scheme used
in \cite{Barrett-Suli}, based on time-discretization, in the case of the Navier--Stokes--Fokker--Planck system with constant viscosity coefficients (or, alternatively, to use a Galerkin approximation scheme in the spatial variables, similar to the one in \cite{FNP}). The added technical difficulties, caused by the presence of the variable viscosity coefficients, can be handled exactly as in \cite{AbFei}, where a similar scheme, based on time-discretization, was applied to a diffuse interface model with viscosity coefficients that depended on the concentration.


\section*{Acknowledgements}
\thispagestyle{empty}

The authors acknowledge the support of the project LL1202 in the programme ERC-CZ funded by the Ministry of Education, Youth and Sports of the Czech Republic.


\end{document}